\documentclass{article}
\usepackage{mathrsfs}
\usepackage{graphicx}
\usepackage{bbm}
\usepackage{relsize}
\usepackage{enumerate}
\usepackage{amssymb,amsmath,amsthm}
\usepackage{bm}
\usepackage[title]{appendix}
\usepackage{url}
\usepackage{marvosym}
\usepackage{authblk}

\usepackage{titlesec}
\usepackage{verbatim}
\usepackage{mathrsfs}
\usepackage[usenames, dvipsnames]{color}
\usepackage[letterpaper, portrait, margin=1in]{geometry}
\usepackage{amsmath, amssymb, amsthm}
\usepackage{graphicx}
\usepackage{mathrsfs}
\usepackage{xcolor}
\usepackage{bbm}
\usepackage{mathtools}
\usepackage[linesnumbered,lined,ruled]{algorithm2e}
\usepackage{algpseudocode}
\usepackage{chngcntr}
\usepackage{float,stmaryrd}

\newcommand{\TTT}{\mathcal{T}}

\newtheorem{theorem}{Theorem}

\newtheorem{prop}[theorem]{Proposition}

\newtheorem{cor}[theorem]{Corollary}
\newtheorem{lemma}[theorem]{Lemma}
\newtheorem{defn}[theorem]{Definition}

\numberwithin{theorem}{section}
\numberwithin{equation}{section}

\titleformat{\section}
{\normalfont\Large\bfseries}{\thesection}{1em}{}
\titleformat{\subsection}
{\normalfont\large\bfseries}{\thesubsection}{1em}{}
\titleformat{\subsubsection}
{\normalfont\normalsize\bfseries}{\thesubsubsection}{1em}{}

\title{Tropically planar graphs}
\author{Desmond Coles, Neelav Dutta, Sifan Jiang, Ralph Morrison, and Andrew Scharf}

\usepackage[utf8]{inputenc}
\begin{document}
\maketitle

\begin{abstract}
We study tropically planar graphs, which are the graphs that appear in smooth tropical plane curves.  We develop necessary conditions for graphs to be tropically planar, and compute the number of tropically planar graphs up to genus $7$.  We provide non-trivial upper and lower bounds on the number of tropically planar graphs, and prove that asymptotically $0\%$ of connected trivalent planar graphs are tropically planar.
\end{abstract}




\medskip

\section{Introduction}

Tropical geometry studies discrete, combinatorial analogs of objects from algebraic geometry.  In the case of an algebraic plane curve, the tropical analog is a \emph{tropical plane curve}, which has the structure of a one-dimensional polyhedral complex, embedded in $\mathbb{R}^2$ in a balanced way.  Each tropical plane curve has an associated \emph{Newton polygon} $\Delta$, which is a lattice polygon.  The curve is dual to a regular subdivision of $\Delta$, as discussed in \cite{ms}; we call the curve \emph{smooth} if that subdivision is a unimodular triangulation.  A smooth tropical plane curve is illustrated in the middle of Figure \ref{figure:full_example}, with its subdivided Newton polygon pictured on the left.  

\begin{figure}[hbt]
   		 \centering
		\includegraphics[scale=1.3]{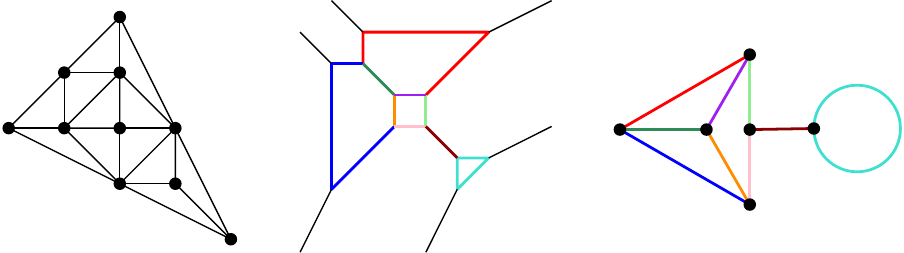}
	\caption{A smooth tropical plane curve in the center, with its subdivided Newton polygon on the left and its skeleton on the right}
	\label{figure:full_example}
\end{figure}

Inside of a smooth tropical plane curve is a graph called its \emph{skeleton}, which is the largest subset onto which the curve admits a deformation retract.   If the polygon has $g$ interior lattice points, then the skeleton has \emph{genus} (that is, first Betti number) equal to $g$.  On the right side of Figure \ref{figure:full_example} we see the skeleton of the smooth tropical curve; it has genus $4$ because the Newton polygon has $4$ interior lattice points.    There is also a natural way to assign lengths to the edges of the skeleton, making it a \emph{metric graph} as in \cite{baker}. In this paper, we will only be concerned with the combinatorial structure of the graph and not the lengths.

\begin{defn} \normalfont
A graph $G$ is said to be  \textit{tropically planar} (or \textit{troplanar} for short) if it is the skeleton of a smooth tropical plane curve.
\end{defn}

Since tropical plane curves are embedded in $\mathbb{R}^2$, all tropically planar graphs are planar.  They are also connected and trivalent since they are dual to unimodular triangulations.  In general, there is no known efficient way to test if a given graph is troplanar, although an algorithm for finding all troplanar graphs of a fixed genus $g$ was designed and implemented in \cite{BJMS}.  In fact, their algorithm went further:  it found all troplanar graphs, and determined which edge lengths were possible on those graphs inside of tropical plane curves.    For genus $g=2,3,4,$ and $5$, they found that there are $2,4,13,$ and $38$ troplanar graphs of genus $g$, respectively. 

In this paper we work to further our understanding of troplanar graphs. We start by developing certain criteria that troplanar graphs must (or must not) satisfy, and by pushing the computations of \cite{BJMS} further to determine the number of troplanar graphs of genus $6$ and genus $7$.  We also determine asymptotic upper and lower bounds on $\mathscr{T}(g)$, the number of troplanar graphs
of genus $g$. Our upper bound in Theorem \ref{theorem:best} shows that $\mathscr{T}(g)=O(2^{\frac{11}{3}g+O(\sqrt{g})})$, which provides one of several proofs that as $g\rightarrow\infty$, most connected trivalent planar graphs are not tropically planar.
Our lower bound in Corollary \ref{cor:upper} shows that $\mathscr{T}(g)=\Omega(\gamma^g)$, where $\gamma\approx 2.47$; this is an improvement on the best previously known result that $\mathscr{T}(g)=\Omega(2^g)$.  There is still a wide gulf between these upper and lower bounds, which will hopefully be narrowed by future research.

Our paper is organized as follows.  In Section \ref{sec:background} we provide necessary background on polygons, graphs, tropical curves, and asymptotics notation. In Section \ref{sec:properties} we present properties of troplanar graphs, reviewing some from previous works as well as developing new ones; we also discuss our computations of troplanar graphs of genus $6$ and $7$ through the lens of these results. In Section \ref{sec:upper} we prove our upper bound on the number of troplanar graphs of a given genus, and in Section \ref{sec:lower} we prove our lower bound.  

\medskip

\noindent \textbf{Acknowledgements.} The authors thank Michael Joswig and Ayush Tewari for helpful comments on an earlier draft of this paper, including finding several errors in the counts at the end of Section \ref{sec:properties}; these errors have been corrected.  The authors are grateful for their support from the 2017 SMALL REU at Williams College, and from the National Science Foundation via Grant DMS1659037.

\section{Background and definitions}
\label{sec:background}

In this section we establish background necessary for stating and proving our results.  This material will cover background on lattice polygons, graphs, and tropical curves, as well as some notation.

\subsection{Lattice polygons}

Any point in $\mathbb{R}^2$ with integer coordinates is called a \emph{lattice point}.  A line segment with lattice endpoints has \emph{lattice length} equal to $1$ less than the number of lattice points on it.  A \emph{lattice polygon} is a polygon whose vertices are lattice points.  Unless otherwise stated, all polygons we consider will be lattice polygons, and will be convex.  Let $\Delta$ be a lattice polygon with $r$ boundary lattice points and $g$ interior lattice points.  We refer to $g$ as the \emph{genus} of the polygon.  It turns out that the numbers $g$ and $r$ encode a great deal of information about the polygon, as illustrated in the following result.

\begin{theorem}[Pick's Theorem]
\label{theorem:picks}
Let $\Delta$ be a lattice polygon with $g$ interior lattice points $r$ boundary lattice points. Then the area of $\Delta$ is given by $$\frac{r}{2}+g-1.$$
\end{theorem}

We say two lattice polygons $\Delta$ and $\Delta'$ are \emph{equivalent} if one is obtained from the other by a matrix transformation $A\in\textrm{PSL}_2(\mathbb{Z})$.  It turns out that if we fix $g\geq 1$, there are only finitely many polygons of genus $g$, up to equivalence.  See \cite[Proposition 2.3]{BJMS} for a discussion of this fact, and see \cite{Castryck2012} for an algorithm to compute all polygons of genus $g$.  The \emph{lattice width} of a polygon $\Delta$ is the width of the smallest horizontal strip containing some polygon $\Delta'$ equivalent to $\Delta$. 

Let $\Delta^{(1)}$ be the convex hull of the $g$ interior lattice points of $\Delta$.  We refer to $\Delta^{(1)}$ as the \emph{interior polygon} of $\Delta$.  If $\Delta^{(1)}$ is a two-dimensional polygon, we say $\Delta$ is \emph{nonhyperelliptic}; otherwise, we say $\Delta$ is \emph{hyperelliptic}.  Thus for a hyperelliptic polygon, $\Delta^{(1)}$ is either the empty set, a single point, or a line segment. Three polygons of genus $3$ are illustrated in Figure \ref{figure:three_polygons}; the first is hyperelliptic, and the other two are nonhyperelliptic.

\begin{figure}[hbt]
   		 \centering
		\includegraphics[scale=0.7]{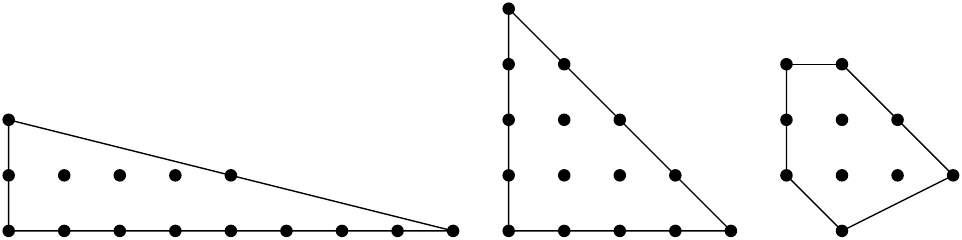}
	\caption{Three lattice polygons; the first is hyperelliptic, and the first two are maximal}
	\label{figure:three_polygons}
\end{figure}

  We say that a lattice polygon $\Delta$ is \emph{maximal} if it is maximal with respect to containment among all polygons with interior polygon $\Delta^{(1)}$, i.e. if there exists no lattice polygon $\Delta'$ properly containing $\Delta$ with $(\Delta')^{(1)}=\Delta^{(1)}$. The first two polygons in Figure \ref{figure:three_polygons} are maximal, while the third is not.   An important tool when studying maximal nonhyperelliptic polygons is \cite[Lemma 2.2.13]{Koelman}, which states that any maximal  nonhyperelliptic polygon $\Delta$ is obtained by ``moving out'' the edges of its interior polygon $\Delta^{(1)}$.  For instance, the middle polygon in Figure \ref{figure:three_polygons} can be obtained by moving out the edges of its interior lattice triangle.  It follows that given any two-dimensional lattice polygon $\Sigma$, either there exists no lattice polygon $\Delta$ whatsoever with $\Delta^{(1)}=\Sigma$, or there exists a unique maximal lattice polygon $\Delta$ with $\Delta^{(1)}=\Sigma$.  See \cite[\S 2]{BJMS} for more discussion.

We now move on to subdivisions of lattice polygons. A \emph{subdivision} of a lattice polygon is a partition of that polygon into lattice subpolygons, such that the intersection of any two subpolygons is a mutual face (either an edge, a vertex, or the empty set).  A \emph{triangulation} is a subdivision where each subpolygon is a triangle.  A triangulation is called \emph{unimodular} if each triangle has the minimum possible area, which by Pick's Theorem is $\frac{1}{2}$. We sometimes call a lattice triangle of area $\frac{1}{2}$ a \emph{unimodular triangle}.  One way to construct a subdivision of $\Delta$ is to assign values $a_{ij}\in \mathbb{R}$ to each lattice point $(i,j)$ in $\Delta$, and to take the convex hull of the points $(i,j,a_{ij})$ in $\mathbb{R}^3$.  The assignment of $a_{ij}$'s to the lattice points is called a \emph{height function}. We then project the lower faces of this polyhedron onto $\Delta$, giving us a subdivision. This process is illustrated on the left in Figure \ref{figure:inducing_regular}. Any subdivision that arises from a height function is called \emph{regular}.  Thus the left triangulation in Figure \ref{figure:inducing_regular} is regular; it turns out that the right triangulation is not regular \cite[Example 2.2.5]{triangulations}.  A \emph{split} in a subdivision is an edge of lattice length one with endpoints on different boundary edges of the lattice polygon. Any split divides a polygon into two lattice polygons.  If both these lattice polygons have positive genus,  we call the split \emph{nontrivial}. For instance, the subdivision in Figure \ref{figure:full_example} has a nontrivial split, separating the triangle into a triangle of genus $1$ and a quadrilateral of genus $3$.

\begin{figure}[hbt]
   		 \centering
		\includegraphics[scale=0.7]{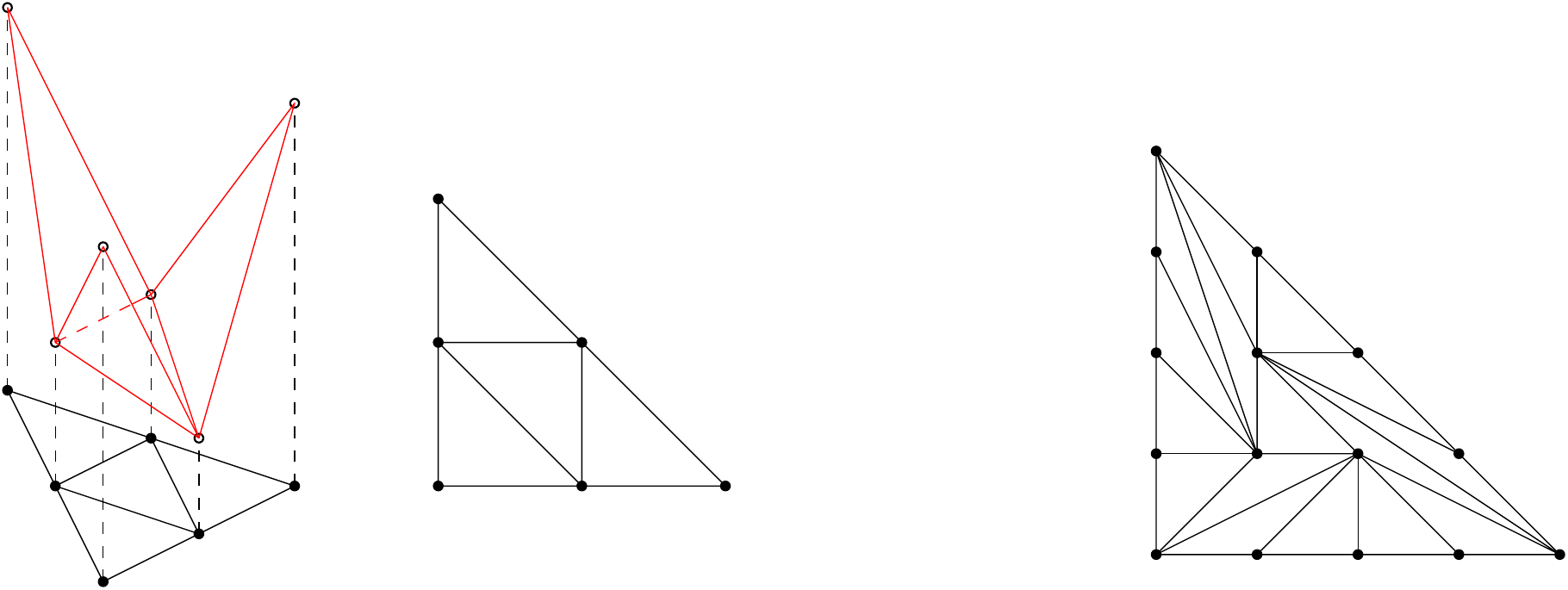}
	\caption{The process of inducing a regular triangulation with a height function; and a nonregular triangulation on the right}
	\label{figure:inducing_regular}
\end{figure}

\subsection{Graphs}  A \emph{graph} $G=(V,E)$ is a finite collection of vertices $V$ joined by a finite collection of edges $E$.  We allow multiple edges between a pair of vertices, and we also allow \emph{loops}, which are edges from a vertex to itself.  We call a graph \emph{connected} if it is possible to move from any vertex to any other vertex using the edges.  We say a graph is \emph{planar} if it can be drawn in $\mathbb{R}^2$ without any edges crossing each other.

The \emph{degree} of a vertex is the total number of edges incident to that vertex, where each loop is counted twice. If a vertex has degree $n$, we will refer to it as an \emph{$n$-valent vertex}. We say that a graph is \emph{trivalent} if every vertex has degree $3$.  An edge $e$ in a graph $G$ is called a \emph{bridge} if the graph $G\setminus\{e\}$ obtained by deleting $e$ from $G$ has more connected components than $G$.  If a graph is connected and has no bridges, we call $G$ \emph{bridgeless}, or equivalently \emph{$2$-edge-connected}.  The connected components that remain after deleting all bridges from a connected graph $G$ are called the \emph{$2$-edge-connected components} of $G$.

An \emph{isomorphism} from $G=(V,E)$ to $G'=(V',E')$ is a bijection $\varphi:V\rightarrow V'$ such that the number of edges between $v,w\in V$ is equal to the number of edges between $\varphi(v),\varphi(w)\in V'$.  If there exists an isomorphism from $G$ to $G'$, we say that $G$ and $G'$ are \emph{isomorphic}.  Virtually every property of a graph is preserved under isomorphism, including the number of bridges and the structure of the $2$-edge-connected components.

The \emph{genus} of a connected graph is $g(G):=|E|-|V|+1$; this is also known as the first Betti number of the graph.  By Euler's Polyhedron Formula, if $G$ is planar then $g(G)$ is the number of bounded regions in any planar drawing of $G$.
The graphs we are most concerned with in this paper are those that are connected and trivalent, with genus $g\geq 2$.  We denote the number of such graphs of genus $g$ as $\mathscr{G}(g)$, and the number of such graphs that are planar as $\mathscr{P}(g)$. The graphs of genus $2$ and $3$ are illustrated in Figure \ref{figure:g2_and_g3}, so $\mathscr{G}(2)=\mathscr{P}(2)=2$ and $\mathscr{G}(3)=\mathscr{P}(3)=5$.  It turns out there is a single nonplanar connected trivalent graph of genus $4$, namely the complete bipartite graph $K_{3,3}$, so $\mathscr{G}(4)=\mathscr{P}(4)+1$.  The connected trivalent graphs were enumerated up to genus $6$ in \cite{balaban}, which found there to be $2$, $5$, $17$, $71$, and $388$ such graphs of genus $2$, $3$, $4$, $5$, and $6$, respectively. We remark that the literature does not always use the term \emph{genus}, and instead stratifies these graphs by the number of vertices; there is no harm in this, since for any connected trivalent graph we have $|V|=2g-2$.

\begin{figure}[hbt]
   		 \centering
		\includegraphics[scale=0.7]{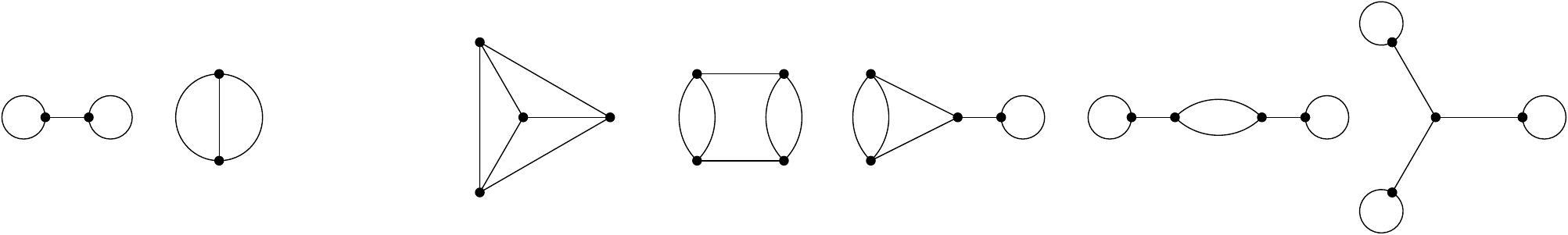}
	\caption{The connected trivalent graphs of genus $2$ and $3$}
	\label{figure:g2_and_g3}
\end{figure}

\subsection{Tropical curves and their skeletons}

We now briefly review the tropical geometry necessary for our paper; see \cite{ms} for more details.  Tropical plane curves are defined by polynomials $p(x,y)$ over the \emph{tropical semiring} $(\mathbb{R}\cup\{\infty\},\oplus,\odot)$, where $a\oplus b=\min\{a,b\}$ and $a\odot b=a+b$.  The subset of $\mathbb{R}^2$ defined by $p(x,y)$ is the set of points where the minimum in the polynomial is achieved at least twice.  By the Structure Theorem \cite[Theorem 3.3.5]{ms}, a   tropical plane curve is a balanced $1$-dimensional polyhedral complex, consisting of edges and rays meeting at vertices.  Forgetting about the embedding into $\mathbb{R}^2$, this means we can interpret a tropical curve as a graph with a $1$-valent vertex at the end of each of the rays.

The \emph{Newton polygon} of a tropical polynomial $p(x,y)$ is the convex hull of all exponent vectors of terms that appear in $p(x,y)$ with non-$\infty$ coefficients.  By \cite[Proposition 3.1.6]{ms},  every tropical plane curve is dual to a regular subdivision of its Newton polygon; in particular, to the regular subdivision induced by the coefficients of the polynomial.  (It follows that every regular subdivision of a lattice polygon has a tropical curve dual to it.)  This duality means that a tropical curves has one vertex for each subpolygon in the subdivision; that two vertices are joined by an edge if and only if the dual subpolygons share an edge; and that there is a ray for each boundary edge of a subpolygon.  

A tropical plane curve is called \emph{smooth} if the corresponding subdivision of its Newton polygon is a unimodular triangulation. Each (finite) vertex is then incident to a total of three rays and vertices.  A smooth tropical plane curve has first Betti number equal to the genus of its Newton polygon.

Now assume that the Newton polygon of tropical curve has genus $g\geq 2$.  The \emph{skeleton} of a smooth tropical plane curve is the graph that is obtained by removing all rays; iteratively retracting any leaves and their edges; and then smoothing over any $2$-valent vertices.  This skeleton is  a connected trivalent planar graph of genus $g$.  As defined in the introduction, any graph that is the skeleton of some smooth tropical plane curve is called \emph{tropically planar}, or simply \emph{troplanar}.  Six smooth tropical plane curves are illustrated in Figure \ref{figure:g2_and_g3_troplanar}, along with skeletons below and dual Newton subdivisions above.  From these examples, we know that the first six graphs in Figure \ref{figure:g2_and_g3} are troplanar; it follows from Proposition \ref{prop:sprawling} that the seventh graph is not.

\begin{figure}[hbt]
   		 \centering
		\includegraphics[scale=1]{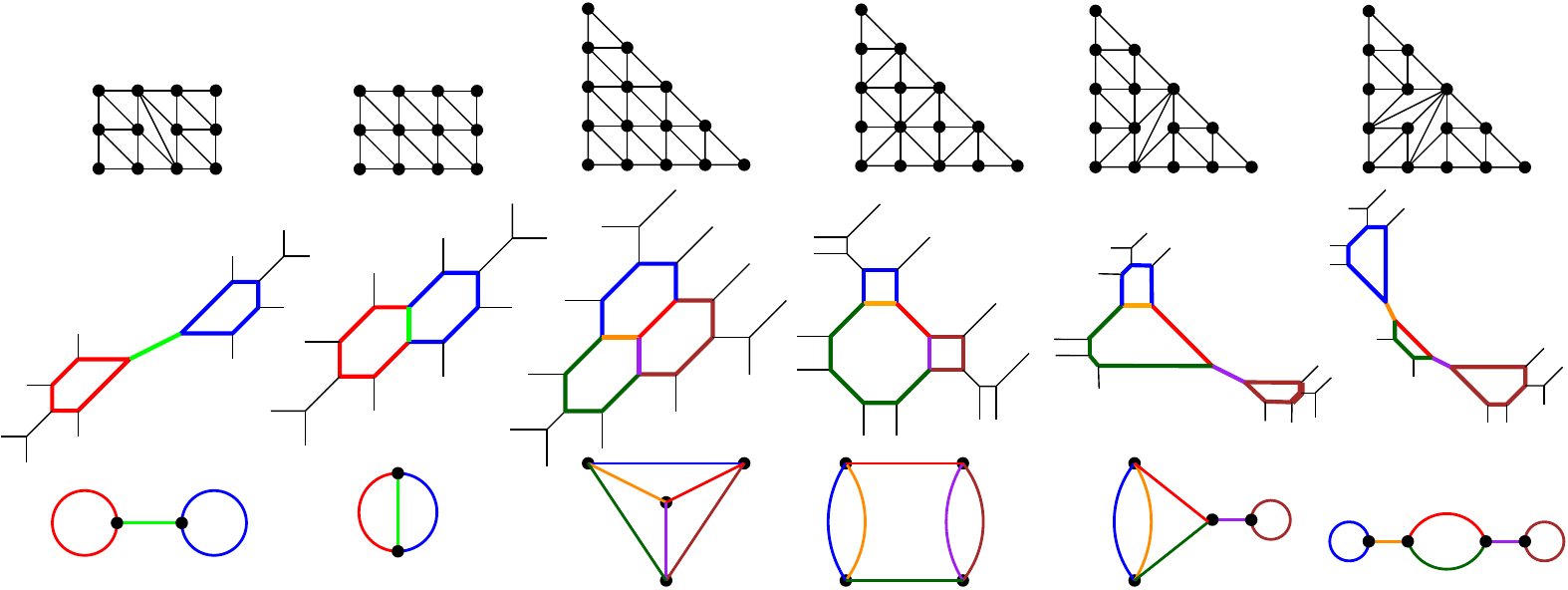}
	\caption{Newton subdivisions, tropical curves, and skeletons}
	\label{figure:g2_and_g3_troplanar}
\end{figure}

The authors of \cite{BJMS} implemented an algorithm for computing all troplanar graphs of a fixed genus, including the achievable edge lengths.  Ignoring this edge length computation, their algorithm can be summarized as follows:

\begin{itemize}
\item[(1)]  Find all lattice polygons with $g$ interior lattice points, up to equivalence.
\item[(2)]  Find all regular, unimodular triangulations of each polygon $P$ found in step (1).
\item[(3)]  For each triangulation found in step (2), compute the dual graph, and find the skeleton.
\end{itemize}
  It is this simplified version of the algorithm that we apply in Section \ref{sec:properties} to find the number of troplanar graphs of genus $6$ and genus $7$.  We also use a key observation made by the authors of \cite{BJMS}, namely that it suffices in step (1) to consider maximal polygons: 
it follows from \cite[Lemma 2.6]{BJMS} that any graph that arises from a nonmaximal polygon of genus $g$ will also arise from a maximal polygon of genus $g$, namely from any maximal polygon containing the original polygon and having the same interior lattice points.

\subsection{Asymptotics notation}  We close this section by briefly recalling big O, little O, and big Omega notation.  Let $f(x)$ and $g(x)$ be functions from $\mathbb{R}$ to $\mathbb{R}$ (although similar notation will hold if the domain is $\mathbb{N}$).  We write $f(x)=O(g(x))$ as $x\rightarrow\infty$ if for all sufficiently large values of $x$, the absolute value of $f(x)$ is at most a positive constant multiple of $g(x)$.  We write $f(x)=o(g(x))$ as $x\rightarrow\infty$ if for all $\varepsilon>0$, there exists $x_0$ such that for $x>x_0$ we have $f(x)\leq \varepsilon g(x)$. There are many different conventions for big Omega notation; for the purposes of this paper, we write $f(x)=\Omega(g(x))$ as $x\rightarrow\infty$ if $g(x)=O(f(x))$ as $x\rightarrow\infty$.  When clear from context, we omit the ``as $x\rightarrow\infty$'' from all these notations.  We write $f(x)\sim g(x)$ when $\lim_{x\rightarrow\infty}\frac{f(x)}{g(x)}=1$.

\section{Properties of troplanar graphs}
\label{sec:properties}

In this section we will discuss some properties and invariants of troplanar graphs. As already  noted, they are connected trivalent planar graphs. In general, they are a difficult family of graphs to study.  For example, the set of troplanar graphs is not minor closed:  the graph of genus $4$ from Figure \ref{figure:minor} is troplanar, as demonstrated by the pictured smooth tropical curve containing it as a skeleton; but the minor obtained from collapsing the central cycle is not troplanar, due to Proposition \ref{prop:sprawling} below.

\begin{figure}[hbt]
   		 \centering
		\includegraphics[scale=1]{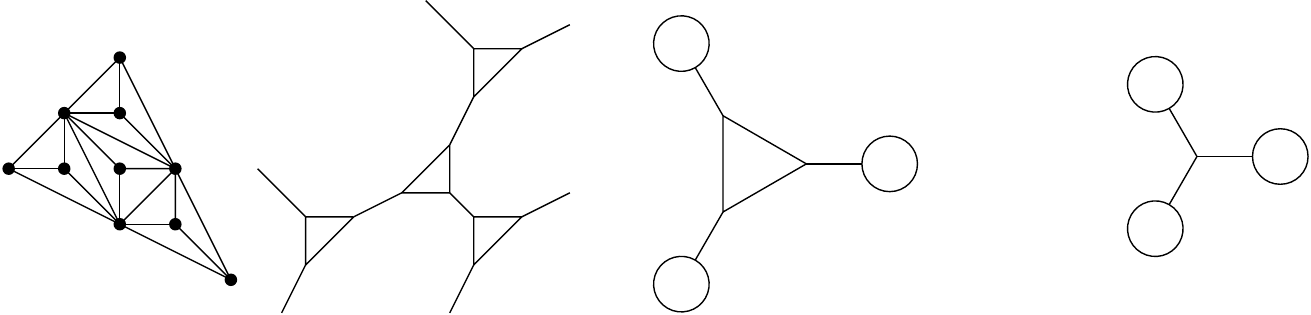}
	\caption{A troplanar graph of genus $4$, with a non-troplanar minor of genus $3$}
	\label{figure:minor}
\end{figure}

We review three previously known criteria for deducing that a graph {fails} to be troplanar. First, if a graph is nonplanar, then it is not troplanar. We say that a connected graph is \emph{sprawling} if it has a vertex such that deleting the vertex from the graph creates three or more components.  We say a planar embedding of a graph is \emph{crowded} if either two faces share two or more edges with one another, or one face shares an edge with itself.  If a graph is planar such that all its planar embeddings are crowded, we say that the graph itself is \emph{crowded}.

\begin{prop}[Proposition 4.1 in \cite{cdmy}]\label{prop:sprawling}
If $G$ is sprawling, then $G$ is not troplanar. \end{prop}
The original proof of this property uses a balancing property of tropical curves.  An alternate proof presented in \cite{BJMS} proves that the required dual structure in a Newton polygon cannot arise in a unimodular triangulation.

\begin{prop}[Lemma 3.5 in \cite{morrisonhyp}]
If $G$ is crowded, then $G$ is not troplanar. \end{prop}
This result follows readily from the fact that no  graph dual to a triangulation can have two faces sharing two edges, or a face that shares an edge with itself.  It is not always immediately obvious is a graph if crowded, since we need information about all planar embeddings of that graph; see \cite[\S 3]{morrisonhyp} for methods of checking whether a graph is crowded.

For the $17$ distinct genus $4$ connected trivalent graphs, these properties (nonplanar, sprawling, and crowded) are enough to rule out all non-troplanar graphs:  $13$ of the graphs are troplanar \cite[\S 7]{BJMS}, while $1$ is nonplanar, and $3$ are sprawling.  (It turns out that no genus $4$ graph is crowded.)

\begin{figure}[hbt]
   		 \centering
		\includegraphics[scale=1]{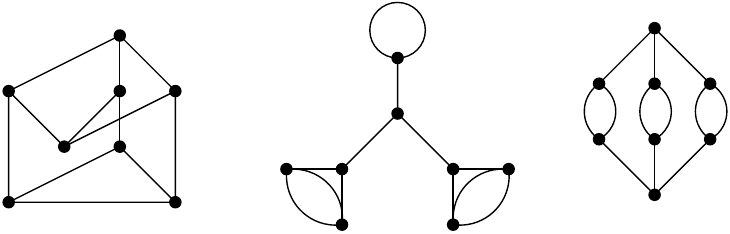}
	\caption{Three genus $5$ graphs easily shown not to be troplanar}
	\label{figure:g5easy}
\end{figure}

Figure \ref{figure:g5easy} illustrates several graphs of genus $5$.  The first is nonplanar, the second is sprawling, and the third is crowded by \cite[Example 3.4]{morrisonhyp}.  Thus, all three are readily determined by our criteria not to be troplanar.  Unfortunately, the results presented thus far are not enough for all graphs of genus $5$:  as computed in \cite{BJMS}, the seven graphs in Figure \ref{figure:g5hard} are not troplanar, even though none are nonplanar, sprawling, or crowded.   

\begin{figure}[hbt]
   		 \centering
		\includegraphics[scale=1]{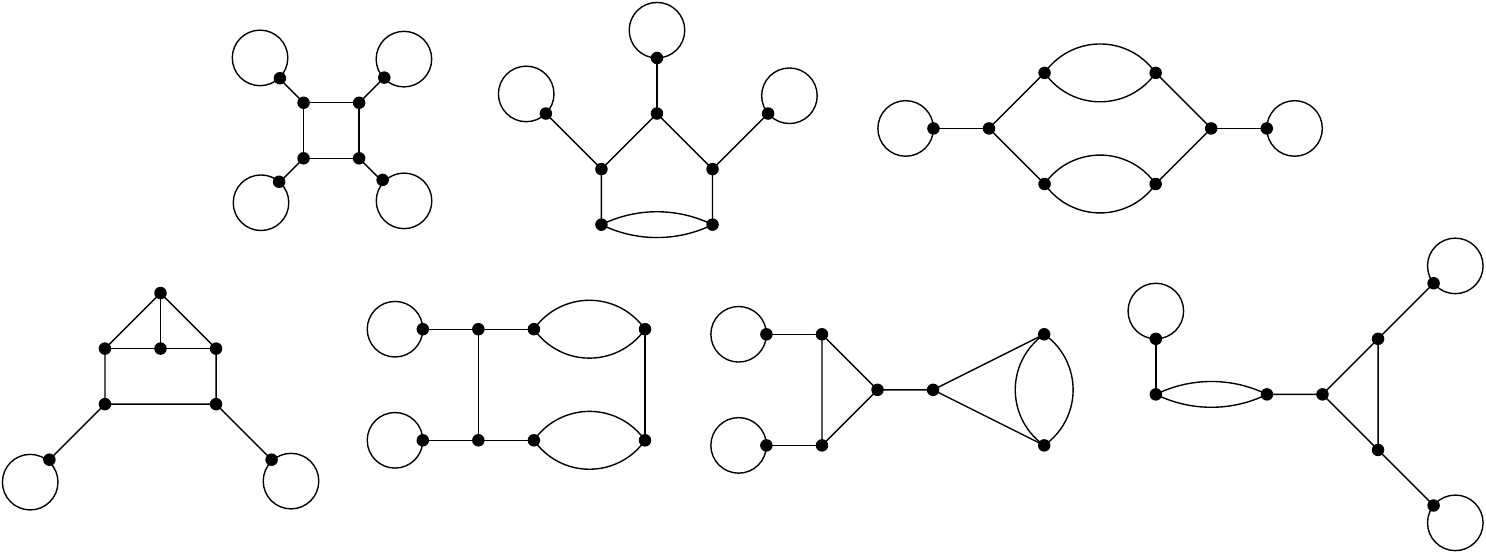}
	\caption{Seven genus $5$ graphs that are not troplanar}
	\label{figure:g5hard}
\end{figure}

We now present new results  to help us further determine which graphs are troplanar and which are not.  The following criterion is similar to the sprawling  criterion in that it provides a  structure that is forbidden for troplanar graphs.
   
\begin{defn}
{ A connected trivalent planar graph is called a \textit{TIE-fighter graph} if it has the form illustrated in Figure \ref{figure:tie}, where each shaded region represent a subgraph of positive genus.}
\end{defn}

\begin{figure}[hbt]
   		 \centering
		\includegraphics[scale=1]{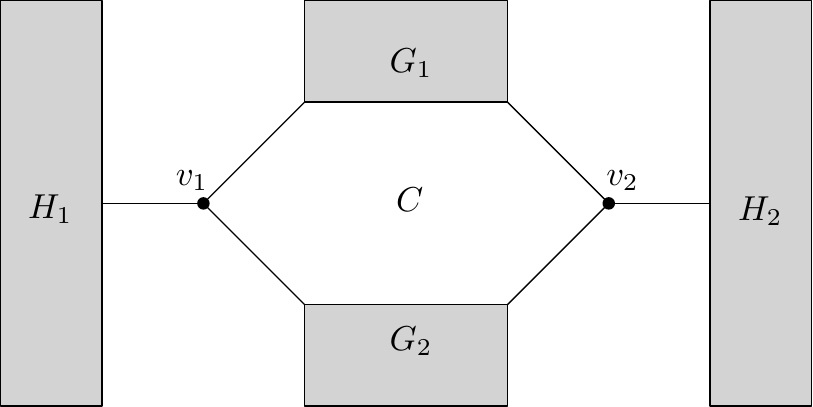}
	\caption{The form of a TIE-fighter graph}
	\label{figure:tie}
\end{figure}

\begin{theorem}\label{theorem:tie}
{Any TIE-fighter graph is not troplanar.}
\end{theorem}

\begin{proof}  Suppose for the sake of contradiction that there exists a TIE-fighter graph $G$ that is troplanar.
Let $\Delta$ be a convex lattice polygon with $\mathcal{T}$ a unimodular triangulation of $\Delta$ such that the corresponding troplanar graph is a TIE-fighter graph. The planar embedding of $G$ coming from $\mathcal{T}$ must have $C$ as a bounded face with the bridges incident to $v_1$ and $v_2$ exterior to it: otherwise the embedding would be crowded. The face formed by $C$ is thus dual to some lattice point $p$ of $\Delta$. Then there must two unimodular triangles, $t_1$, and $t_2$, where $t_i$ is dual to the vertex $v_i$ in the bridge $e_i$ connected to $H_i$. Since $v_1$ and $v_2$ lie on the cycle $C$ in $G$, the triangles $t_1$ and $t_2$ in $\mathcal{T}$ must intersect in a shared vertex, namely the lattice point $p$ dual to the face $C$. Moreover, and the edges of $t_1$ and $t_2$ not intersecting  must be nontrivial splits of $\Delta$, since they yield the bridges $e_1$ and $e_2$. Without loss of generality let $p$ be at the origin. Let these splits be $s_1$ and $s_2$ (so that $s_i$ is an edge of $t_i$), and let $l_i$ be the line passing through $s_i$.

\begin{figure}[hbt]
   		 \centering
		\includegraphics[scale=.8]{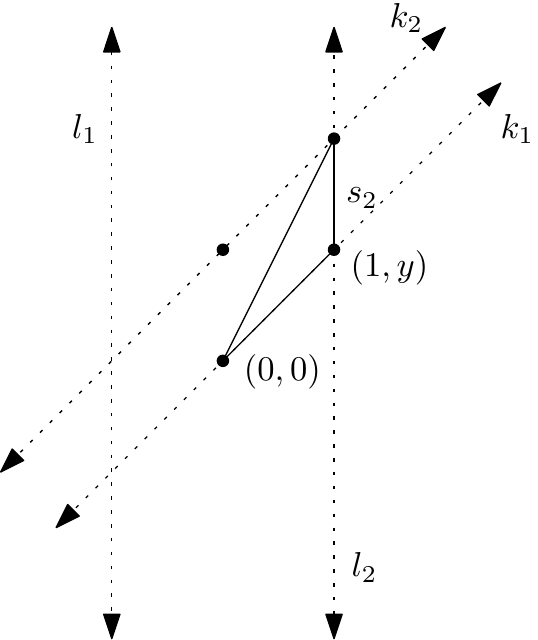}\quad\quad 	\includegraphics[scale=.8]{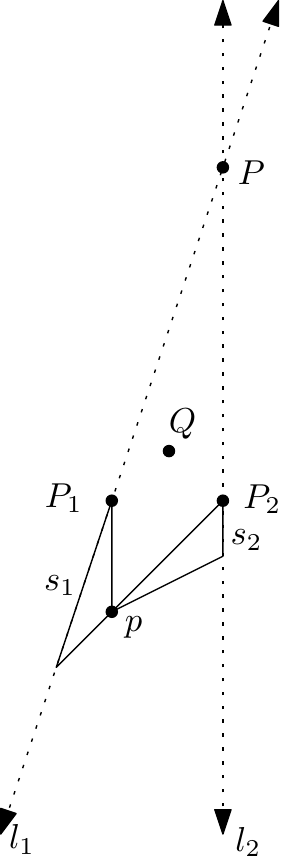}
	\caption{Illustrations for the two cases of our proof}
	\label{figure:tie_cases}
\end{figure}

We now split into two cases, both illustrated in Figure \ref{figure:tie_cases}.  First, assume $l_1$ and $l_2$ are parallel. After a change of coordinates we may assume both $l_1$ and $l_2$ are vertical. Thus any lattice point in $s_1$ has $x$-coordinate -1, and any lattice point in $s_2$ has $x$-coordinate 1; otherwise $t_1$ and $t_2$ would not have area $1/2$ each. Note that $\Delta$ must contain the lattice points $(0,1)$ and $(0,-1)$, because $G_1$ and $G_2$  have genus greater than zero, and any  lattice point in $\Delta$ corresponding to a bounded face from some $G_i$ must be contained in the same connected component of $\Delta\setminus (s_1\cup s_2$) as the origin. Now let the lattice point on $s_2$ with the lower y-coordinate be $(1,y)$; perhaps applying a reflection over the $x$-axis, we may assume $y\geq 0$. Because $\Delta$ is convex and $s_2$ is a split, we now observe that any lattice point dual to some cycle of $H_2$ must be strictly contained in the region bounded by the lines $k_1$ and $k_2$, where $k_1$ passes through $(1,y)$ and the origin, and $k_2$ passes through $(0,1)$ and $(1,y+1)$. But there are no lattice points in the interior of this region, a contradiction.

For the second case, assume $l_1$ and $l_2$ are not parallel. 
Let $P$ be the intersection point of $l_1$ and $l_2$. Without loss of generality we may assume $P$ lies above the $x$-axis. Let $P_i$ be the lattice point in $s_i$ nearest $P$.  Since  $G_1$ and $G_2$ both have genus greater than $0$, there must be some interior lattice point $Q$ of $\Delta$ contained in the interior of the convex hull of $p=(0,0)$, $P_1$, $P_2$, and $P$. We have that $Q$ is closer to the intersection point $P$ than $p$ is, and it follows that $Q$ has Euclidean distance to either $l_1$ or $l_2$ strictly smaller than that of $p$ to that line. This means that for some $i$, the lattice triangle formed by the convex hull of $s_i$ and $Q$ has strictly smaller area than $t_i$. However, this is impossible since $t_i$ has the minimum possible area of $1/2$, a contradiction.

Having reached a contradiction is both cases, we conclude that any TIE-fighter graph is not troplanar.
\end{proof}

The first three graphs in Figure \ref{figure:g5hard} are all TIE-fighter graphs, so Theorem \ref{theorem:tie} provides a proof that they are not troplanar.  We also immediately obtain the following corollary, which forbids a graph from having too many loops in a row.

\begin{cor}\label{corollary:three_loops}
Let $G$ be a troplanar graph with of genus $g\geq 5$.  Then $G$ cannot have three vertices on a path, each incident to an edge that is incident to a loop.
\end{cor}

\begin{proof}
Suppose $G$ is a troplanar graph with three vertices in a path, each incident to an edge that is incident to a loop.  This path must be part of a cycle:  otherwise the graph would be sprawling, as removing the middle vertex would disconnect the graph into three components. Thus, $G$ must have the form illustrated on the left in Figure \ref{figure:three_loops}.  If the boxed area has positive genus, then $G$ is a TIE-fighter graph as illustrated on the right of the figure, and so cannot be troplanar.  Thus the boxed area must have genus $0$, meaning that $G$ is the genus $4$ graph from Figure \ref{figure:minor}.  In particular, we cannot have $g\geq 5$.
\end{proof}

\begin{figure}[hbt]
   		 \centering
		\includegraphics[scale=0.6]{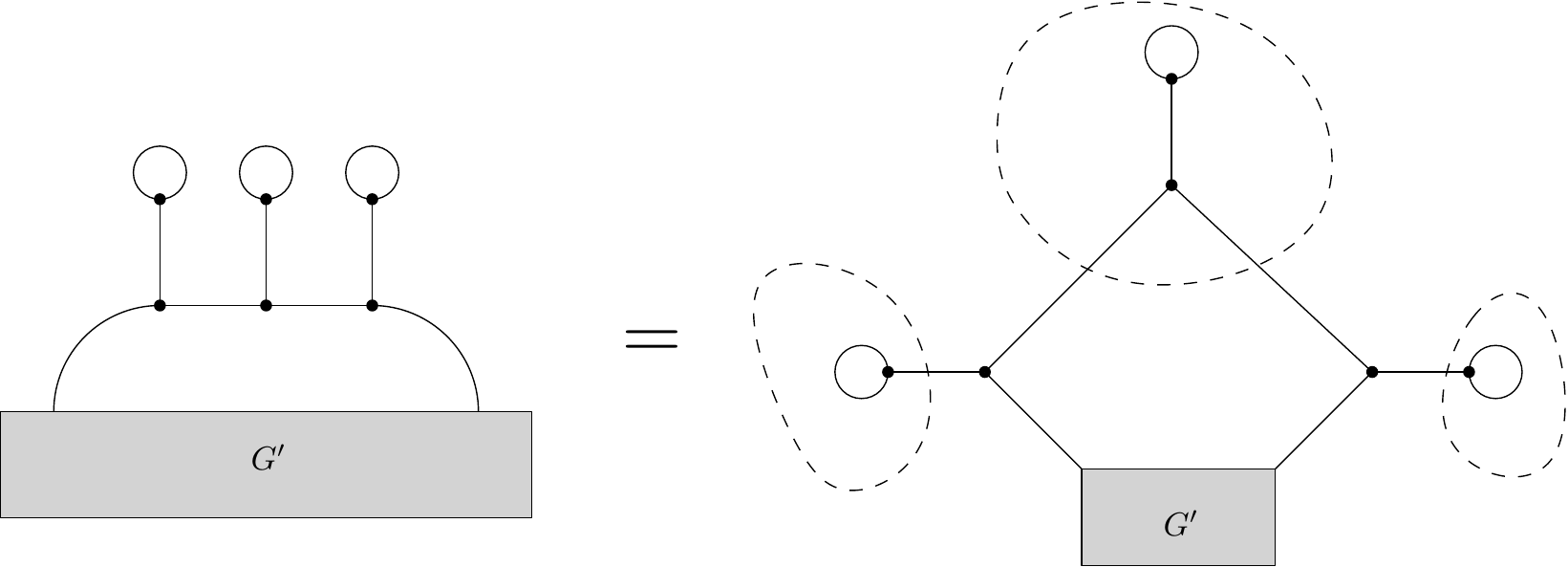}
	\caption{A graph that is a TIE-fighter if $g\geq 5$}
	\label{figure:three_loops}
\end{figure}

We now present two more  results, both involving bridges, that relate the troplanarity of different graphs to one another.

\begin{prop}\label{prop:bridge_deletion}
Let $G$ be a troplanar graph with a bridge $e$. Then the connected components of $G\setminus \{e\}$, after smoothing over $2$-valent vertices, are troplanar.
\end{prop}

By convention, if one of the components is simply a loop, we consider that loop to be troplanar.

\begin{proof}
Let $G$ be a troplanar graph with a bridge $e$, and let $\Delta$ be a Newton polygon with a regular unimodular triangulation $\mathcal{T}$ giving rise to $G$.
The bridge in a troplanar graph $G$ corresponds to a split $\mathcal{T}$. The split subdivides $\Delta$ into two sub polygons, $\Delta_1$ and $\Delta_2$. Since the starting triangulation was regular, the resulting triangulations $\mathcal{T}_1$ and $\mathcal{T}_2$ of $\Delta_1$ and $\Delta_2$ are regular as well:  the height function that induced the triangulation on $\Delta$ can be restricted to each $\Delta_i$ to give the restricted triangulations.  The connected components $G_1$ and $G_2$ of $G\setminus\{e\}$ are then troplanar since they arise from $\mathcal{T}_1$ and $\mathcal{T}_2$.    This is illustrated in Figure \ref{figure:split}.
\end{proof}

\begin{figure}[hbt]\label{figure:split}
   		 \centering
\includegraphics[width=\textwidth]{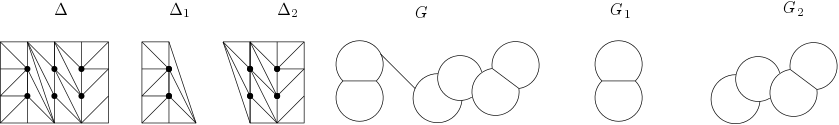}
	\caption{Splitting a polygon into smaller polygons, and splitting a troplanar graph into smaller troplanar graphs }
\end{figure}

This allows us to construct families of non-troplanar graphs of arbitrarily high genus that are not nonplanar, sprawling, crowded, or TIE-fighters graphs. In particular, we can take one of the four graphs on the bottom of Figure \ref{figure:g5hard}, add a bridge $e$ on any edge that is not already a bridge, and attach any troplanar graph of any genus.  This graph cannot be troplanar, since removing $e$ yields two graphs, one of which is not troplanar.

We now present a graph theoretic surgery on bridges for preserving troplanarity. This move is inspired by the well-known bistellar flip in a triangulation. Unfortunately bistellar flips do not in general preserve the regularity of triangulations, but we make use of them in a case where they do.

 Let $b$ be a bridge of a troplanar graph $G$ with end vertices $v$ and $w$, as illustrated in Figure \ref{figure:bridgereduction}. Let the other two edges (possibly non-distinct in the case of a loop) emanating from $v$ be $e_1$, and $e_2$, and let $f_1$, and $f_2$ be the other edges emanating from $w$. Delete $b$, so that  $v$ and $w$ are both $2$-valent. Split $v$ into two distinct univalent vertices $v_1$ and $v_2$, with $v_i$ a $1$-valent vertex incident to $e_i$. Do the same for $w$. Identify $v_1$ with $w_1$ to make a new vertex $v'$, and $v_2$ with $w_2$ to make $w'$, and add an edge $b'$  between $v'$ and $w'$.  This process is illustrated in Figure \ref{figure:bridgereduction}.  We say the new graph is the \textit{bridge reduction of $G$ with respect to $b$.}

\begin{figure}[hbt]\label{figure:bridgereduction}
   		 \centering
\includegraphics[scale=0.8]{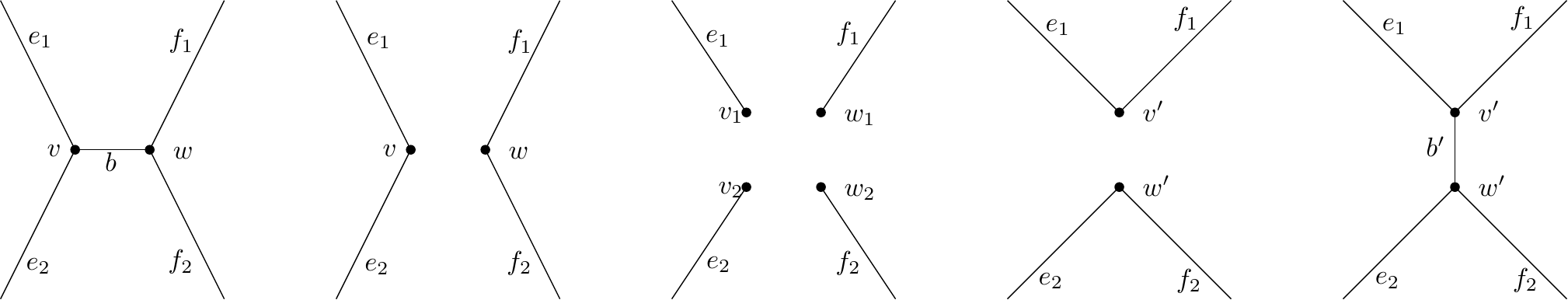}
	\caption{Bridge reduction}
\end{figure}

\begin{prop}\label{prop:bridge_reduction}
Let $G$ be a troplanar graph with a bridge $b$. Then the bridge reduction of $G$ with respect to $b$ is troplanar.
\end{prop}

\begin{proof}  For a height function $\omega$, write $\mathcal{T}(\omega)$ for the triangulation induced by $\omega$.  Let $\Delta$ be a lattice polygon and $\omega$ a height function such that the triangulation $\mathcal{T}(\omega)$ is a unimodular triangulation of $\Delta$ giving rise to $G$.  The bridge in $G$ is dual to a split in $\TTT(\omega)$.  Let $a$ and $b$ be the endpoints of the split, and let $c$ and $d$ be the other points in the  two triangles containing the split. We will refer to the quadrilateral containing these triangles be $abcd$. Since $a$ and $b$ form a split and since $\Delta$ is convex, we know that $abcd$ is convex.  The surgery for bridge reduction corresponds to a bistellar flip in $abcd$; that is, it corresponds to removing the split from $a$ to $b$,  subdivide $abcd$ with a segment from $c$ to $d$. This is illustrated in Figure \ref{figure:bistellar}.  All that remains to show is that the triangulation obtained from this bistellar flip is still regular.

\begin{figure}[hbt]
   		 \centering
\includegraphics[scale=1]{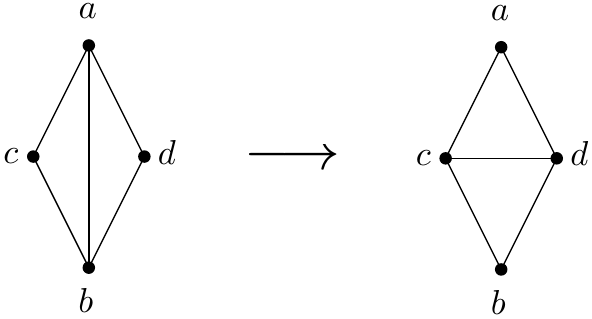}
	\caption{A bistellar flip}
	\label{figure:bistellar}
\end{figure}

The split between $a$ and $b$ subdivides $\Delta$ into two polygons, $\Delta_1$ and $\Delta_2$.  Let $\omega_1=\omega|_{\Delta_1}$ and $\omega_2=\omega|_{\Delta_2}$ be the restrictions of $\omega$ to these polygons.  Note that the triangulation $\TTT(\omega_i)$ is simply the triangulation $\TTT(\omega)$ restricted to $\Delta_i$. 


\begin{figure}[hbt]
   		 \centering
\includegraphics[scale=0.8]{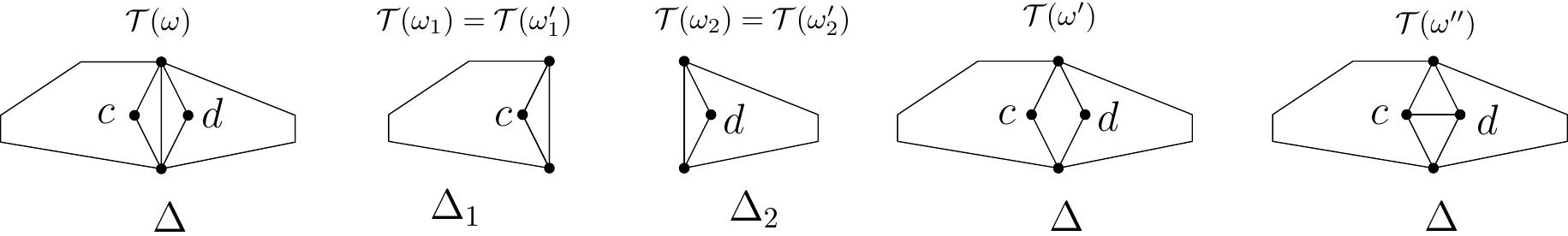}
	\caption{The starting triangulation of $\Delta$; the restricted triangulations of $\Delta_1$ and $\Delta_2$; and the next two triangulations of $\Delta$}
	\label{figure:splitting_polygon}
\end{figure}

By \cite{triangulations}, we may choose a height function $\omega_1'$ on $\Delta_1$ that is identically $0$ on $abc$ such that $\mathcal{T}(\omega_1)=\mathcal{T}(\omega_1')$.  Similarly we may choose $\omega_2'$ on $\Delta_1$ that is identically $0$ on $abd$ such that $\mathcal{T}(\omega_2)=\mathcal{T}(\omega_2')$.  It follows that $\omega_1' > 0$ and $\omega_2' > 0$ on all lattice points of $\Delta_1$ and $\Delta_2$, respectively, outside of $abcd$. Since $\omega_1'$ and $\omega_2'$ agree on $a$ and $b$, we can glue them together to obtain a height function $\omega'$ on $\Delta$.  Because of the positivity of all coefficients away from $abcd$, we have that $\mathcal{T}(\omega')$ is identical to $\mathcal{T}(\omega)$, except that instead of the triangles $abc$ and $abd$ we have the quadrilaterial $abcd$. We can the perform a pulling refinement \cite[\S 16.2]{handbook} by pulling at the lattice point $c$ (or $d$) to obtain  a regular triangulation $\mathcal{T}(\omega'')$ which has an edge from $c$ to $d$, and which is otherwise identical to $\mathcal{T}(\omega')$.  This regular triangulation is precisely the triangulation obtained by performing our bistellar flip, thus completing the proof.
\end{proof}


Proposition \ref{prop:bridge_reduction} helps us relate troplanarity of graphs of the same genus.  For instance, we know that if the seventh graph in Figure \ref{figure:g5hard} were troplanar, then so would the sixth graph.  Contrapositively, if we take it as a given that the sixth graph is not troplanar, then we know the same is true of the seventh.  We can also use the proposition to bound the number of troplanar graphs by the number of $2$-edge-connected troplanar graphs.

\begin{cor}\label{corollary:two-connected}
Let $\mathscr{T}(g)$ be the number of troplanar graphs of genus $g$, and  let $\mathscr{T}^{(2)}(g)$ be the number of $2$-edge-connected troplanar graphs of genus $g$.  Then $\mathscr{T}(g)\leq 2^{g-1}\mathscr{T}^{(2)}(g)$.
\end{cor}

\begin{proof}  Let $G$ be a troplanar graph of genus $G$.  Iteratively reduce all its bridges to obtain a $2$-edge-connected troplanar graph $G'$.  The order in which the bridge reductions are performed does not matter, so $G'$ is well-defined.

We now ask how many distinct graphs could be bridge-reduced to the same graph $G'$. Let us assume that $G'$ is troplanar and $2$-edge-connected.  Let $\mathcal{T}$ be a regular unimodular triangulation of a lattice polygon $\Delta$ giving rise to $G'$.  Viewing $\mathcal{T}$ as a subgraph whose vertices are $\Delta\cap \mathbb{Z}^2$, we then have that $G'$ is the dual graph of the subgraph $H$ of $\mathcal{T}$ induced by the interior lattice points of $\Delta$.  Every $2$-edge-cut of $G'$ corresponds to a bridge in $H$.  Any graph $G$ giving rise to $G'$ via a sequence of bridge reductions can be recoverd by choosing some subset of the bridges in $H$, and reversing the bridge reduction surgery at each corresponding $2$-edge-cut of $G'$.  Since $H$ is a  graph with $g$ vertices, it has at most $g-1$ bridges, so there are at most $2^{g-1}$ subsets of bridges of $H$.  It follows that no more than $2^{g-1}$ graphs could give rise to $G'$ via bridge reductions.

Since every troplanar graph can be bridge reduced to a $2$-edge-connected troplanar graph, we have $\mathscr{T}(g)\leq 2^{g-1}\mathscr{T}^{(2)}(g)$.
\end{proof}





We close this section by presenting some results on the troplanar graphs for genus $6$ and $7$. To find all such troplanar graphs, we must first determine the maximal nonhyperelliptic polygons of genus $6$ and $7$.

\begin{figure}[hbt]
\centering
\includegraphics[scale=0.8]{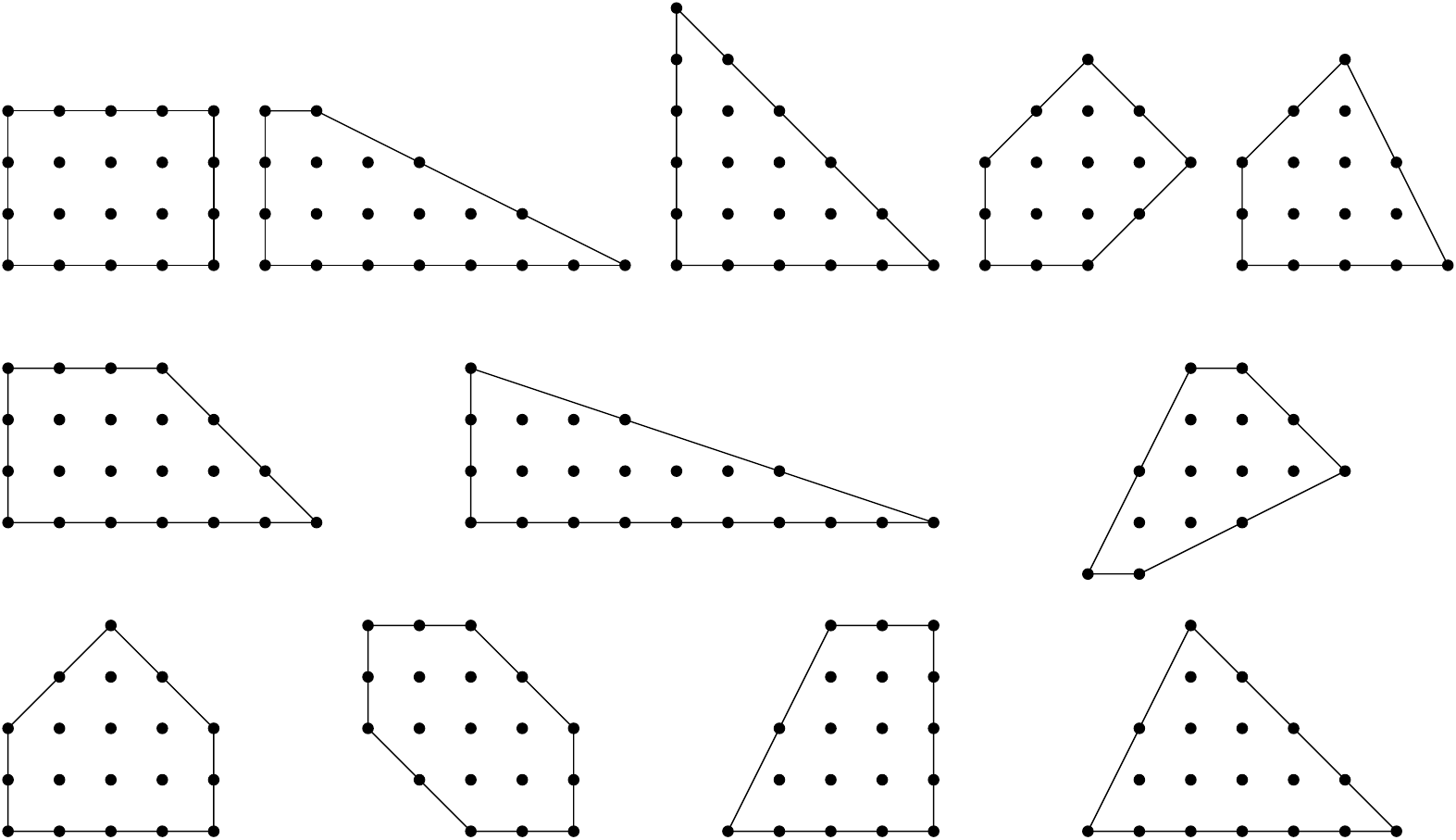}
\caption{The maximal nonhyperelliptic polygons of genus $6$ and genus $7$}
\label{figure:genus_6_7_polygons}
\end{figure}

\begin{prop}\label{prop:maximal67}
The maximal nonhyperelliptic polygons of genus $6$ and genus $7$ are those pictured in Figure \ref{figure:genus_6_7_polygons}.
\end{prop}

\begin{proof}
The main tool we use here is \cite[Lemma 2.2.13]{Koelman}, which states that any maximal  nonhyperelliptic polygon $\Delta$ is obtained by moving out the edges of its interior polygon $\Delta^{(1)}$.  Thus we need only consider all $2$-dimensional lattice polygons with $6$ or $7$ lattice points, and determine which can be pushed-out to obtain another lattice polygon.

First we will argue that none of the interior polygons can be nonhyperelliptic themselves.  Any nonhyperelliptic polygon has lattice width at least $3$, and so any polygon with nonhyperelliptic interior polygon has lattice width at least $5$ by  \cite[Theorem 4]{castryckcools}.  But the only polygon of genus $6$ with lattice width $5$ is the triangle with vertices at $(0,0)$, $(5,0)$, and $(0,5)$, whose interior polygon has genus $0$; and as computed in \cite[Table 1]{Castryck2012}, no lattice polygon of genus $7$ has lattice width $5$.  Thus, no maximal polygon of genus $6$ or $7$ has a nonhyperelliptic interior polygon.

Thus we need only consider candidate interior polygons that are hyperelliptic. We now use a classification of hyperelliptic lattice polygons $\Delta$ of genus $g\geq 0$ presented in  \cite{Koelman} and reproduced in \cite[Theorem 10]{Castryck2012}.  If the interior polygon has genus $0$, it must either be a right trapezoid of height $1$, or (in the case of $6$ lattice points) a triangle with side lengths $2$.  These cases yield five polygons in Figure \ref{figure:genus_6_7_polygons}, namely the first three polygons of genus $6$ and the first two polygons of genus $7$.  If the interior polygon has genus $1$, then it must be one of the $16$ genus $1$ polygons pictured in \cite[Theorem 10(b)]{Castryck2012}. Two of these have $6$ lattice points, and four of these have $7$ lattice points. It turns out every such polygon can be pushed out, yielding the other two pictured polygons of genus $6$, and the bottom row of polygons of genus $7$.

Finally we deal with the case that the interior polygon has genus at least $2$.  As it is hyperelliptic, it must have one of the forms classified in \cite[Theorem 10(c)]{Castryck2012}.  For a genus $6$ polygon, the interior polygon will either have genus $2$ or $3$, and for genus $7$, it will either have genus $2$ or $3$ or $4$ (since any lattice polygon has at least three boundary points).  Running through the finitely many cases shows that there is a unique such polygon with $6$ or $7$ lattice points that can be pushed out to a lattice polygon: it has $7$ lattice points, and yields the final polygon of genus $7$ pictured on the right of the middle row in Figure \ref{figure:genus_6_7_polygons}.
\end{proof}

We used TOPCOM \cite{TOPCOM} to find all regular unimodular triangulations of these maximal nonhyperelliptic polygons, and then computed the resulting troplanar graphs.  We would have also included the graphs arising from the hyperelliptic polygons, although it turns out all such graphs also arose from nonhyperelliptic polygons for $g=6$ and $g=7$.
In the end we found that there are $152$ troplanar graphs of genus $6$, and $672$ troplanar graphs of genus $7$. The complete lists of these graphs are available at \url{https://sites.williams.edu/10rem/supplemental-material/}. Table \ref{table:graph_numbers} contains, for genus $g$ from $2$ to $7$, the numbers of troplanar graphs $\mathscr{T}(g)$,  of connected trivalent planar graphs $\mathscr{P}(g)$, and  of connected trivalent graphs $\mathscr{G}(g)$. (To our knowledge, the value of $\mathscr{P}(7)$ is not present in the literature, and so has been omitted.) We remark that the sequence $2,4,13,38,152$ does not match any sequence on the Online Encyclopedia of Integer Sequences.

\begin{table}[hbt]
\begin{center}
\begin{tabular}{ |c||c|c|c| } 
 \hline
 $g$ & $\mathscr{T}(g)$ & $\mathscr{P}(g)$ &$\mathscr{G}(g)$ \\
  \hline
 $2$& $2$& $2$& $2$\\
 \hline
  $3$& $4$& $5$& $5$\\
 \hline
  $4$& $13$& $16$& $17$\\
 \hline
  $5$& $38$& $67$& $71$\\
 \hline
  $6$& $152$& $354$& $388$\\
 \hline
  $7$& $672$& $?$& $2592$\\
 \hline
\end{tabular}
\end{center}
\caption{The number of troplanar, planar, and general graphs that are connected and trivalent, by genus}
\label{table:graph_numbers}
\end{table}

\begin{figure}[hbt]
\includegraphics[width=\textwidth]{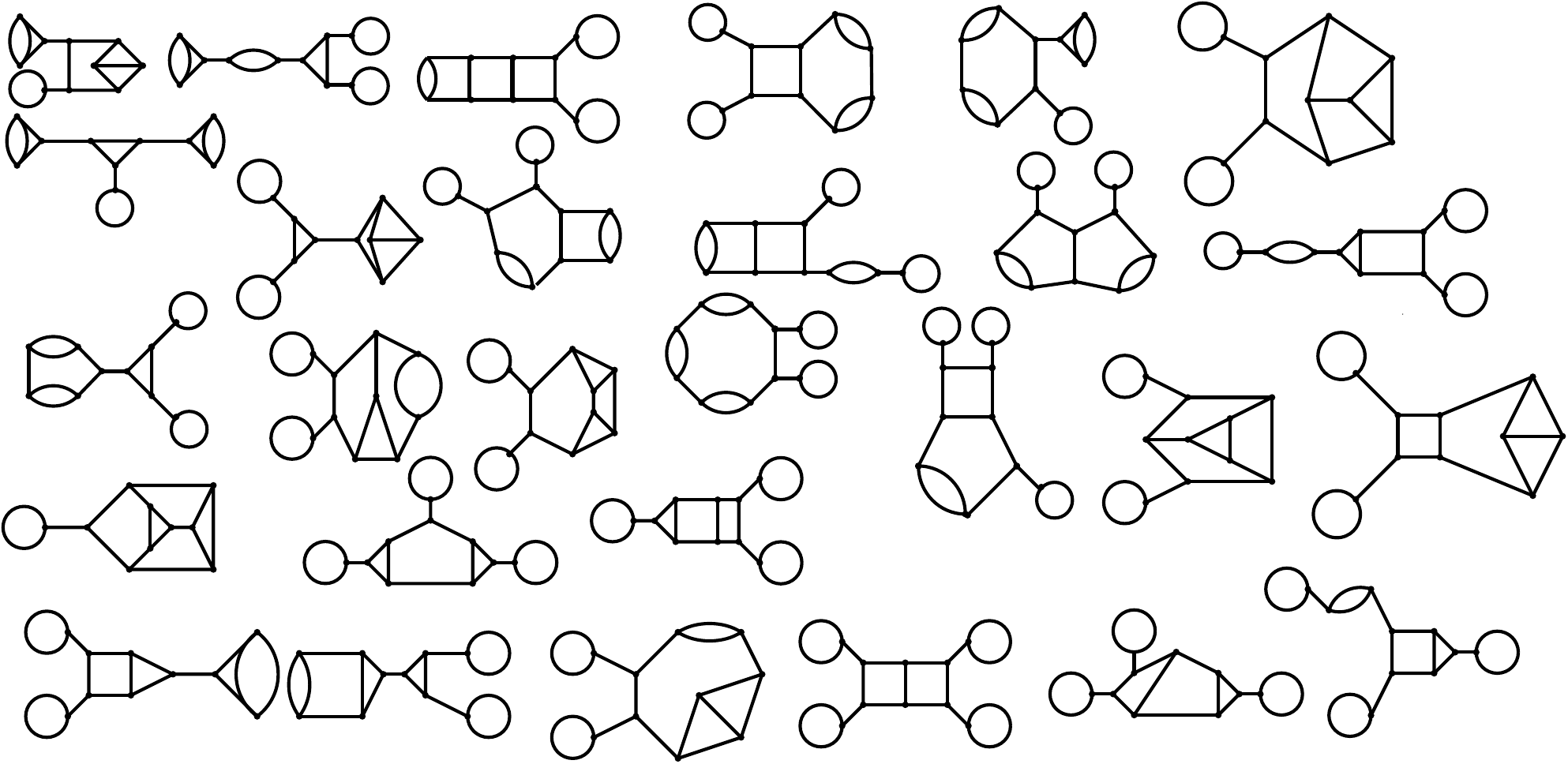}
\caption{All non-troplanar genus 6 graphs which are not ruled by any known  criterion}
\label{figure:genus6_mysteries}
\end{figure}

To get a feel for how the results from earlier in this section rule out non-troplanar graphs, we review some data on graphs of genus $6$, all of which are illustrated in \cite{balaban}.  Of the 388 connected trivalent graphs of genus 6, a total of $152$ are troplanar.  This leaves $236$ non-troplanar  graphs, of which
$34$ are nonplanar.  Of the remaining $202$ planar-but-not-troplanar graphs, $87$ are sprawling and $63$ are crowded (with $5$ both sprawling and crowded).  This gives us $57$ non-troplanar graphs not ruled out by any criterion known prior to this paper.  With our Theorem \ref{theorem:tie}, we can rule out $19$ more; and with Proposition \ref{prop:bridge_deletion}  we rule out another $10$.  This leaves $28$ graphs that are not troplanar, but are not ruled out by any known criterion; these graphs are pictured in Figure \ref{figure:genus6_mysteries}.

We remark that none of the troplanar graphs of genus $6$ have two adjacent vertices where each is also incident to a bridge that is incident to a loop.  If one could prove that a result akin to Corollary \ref{corollary:three_loops} forbidding two loops in a row for graphs with $g\geq 6$, this would rule out an additional $18$ graphs of genus $6$.

\section{An upper bound on the number of troplanar graphs}
\label{sec:upper}

Let $g\geq 2$, and let us consider the proportion of connected trivalent planar graphs of genus $g$ that are troplanar; that is, let us consider  $\mathscr{T}(g)/\mathscr{P}(g)$. From Table \ref{table:graph_numbers}, we have $\mathscr{T}(2)/\mathscr{P}(2)=2/2=1$, $\mathscr{T}(3)/\mathscr{P}(3)=4/5=0.8$, and so on; in other words, $100\%$ of genus $2$ planar graphs are troplanar, and $80\%$ of genus $3$ planar graphs are troplanar. We begin this section by proving that $\mathscr{T}(g)/\mathscr{P}(g)$ tends to $0$ as $g$ tends to infinity.  We could also consider the number of troplanar graphs that are \emph{simple} (that is, have no loops or multiedges) compared to the number of simple connected planar graphs.  Calling these numbers $\mathscr{T}^{\textrm{s}}(g)$ and  $\mathscr{P}^{\textrm{s}}(g)$, we will also prove that $\mathscr{T}^{\textrm{s}}(g)/\mathscr{P}^{\textrm{s}}(g)$ tends to $0$ as $g$ tends to infinity. We will then develop a stronger asymptotic upper bound on $\mathscr{T}(g)$, which will also imply that $\mathscr{T}(g)/\mathscr{P}(g)$ tends to $0$.

\begin{theorem}[Theorem 5 in \cite{randomcubic2}]\label{theorem:random_planar_cubic} Let $H$ be a fixed connected planar graph with one vertex of degree $1$ and each other vertex of degree $3$.  Then there exists $\delta>0$ such that for all $g$, the probability that a random connected trivalent planar graph of genus $g$ contains fewer than $\delta g$ copies of $H$ is $e^{-\Omega(g)}$.
\end{theorem}

In fact, a much stronger result is true:  as shown in \cite[]{randomcubicfinal}, the number of copies of $H$ approaches a certain normal distribution as $g$ goes to infinity. We note that these result were stated in terms of $n$, the number of vertices of the graph, rather than in terms of $g$.  However, for a connected trivalent graph with $n$ vertices we have $n=2g-2$, so the result is easily translated into $g$.  We also remark that the original result was stated for simple graphs; it is easily generalized to multigraphs.

\begin{theorem}\label{theorem:most_not_troplanar}
We have $\lim_{g\rightarrow\infty}\frac{\mathscr{T}(g)}{\mathscr{P}(g)}=0$, and $\lim_{g\rightarrow\infty}\frac{\mathscr{T}^{\textrm{s}}(g)}{\mathscr{P}^{\textrm{s}}(g)}=0$.
\end{theorem}

\begin{proof} We will prove that $\lim_{g\rightarrow\infty}\frac{\mathscr{T}(g)}{\mathscr{P}(g)}=0$; the other result follows from an identical argument, since the graph $H$ below is simple.  Let $H$ be the sprawling graph illustrated in Figure \ref{figure:graph_h}. Any trivalent graph containing a copy of $H$ must also be sprawling, with $v$ serving as a disconnecting vertex. By Theorem \ref{theorem:random_planar_cubic},  there exists $\delta>0$ such that for all $g$, the probability that a random connected trivalent planar graph of genus $g$ contains fewer than $\delta g$ copies of $H$ is $e^{-\Omega(g)}$.  For sufficiently large $g$, we have $\delta g\geq 1$.  This means that the probability a graph does not contain any copy of $H$ goes to $0$ as $g\rightarrow\infty$, which implies that the probability a graph is troplanar goes to $0$ as $g\rightarrow\infty$.  We conclude that $\lim_{g\rightarrow\infty}\frac{\mathscr{T}(g)}{\mathscr{P}(g)}=0$.
\end{proof}

\begin{figure}[hbt]
   		 \centering
		\includegraphics[scale=1]{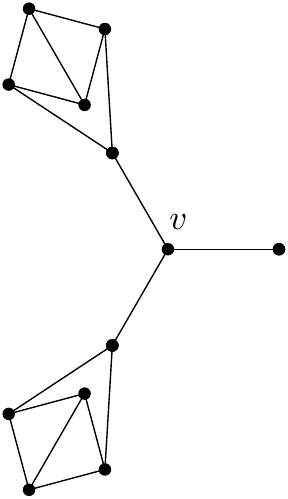}
	\caption{The graph $H$ used in the proof of Theorem \ref{theorem:most_not_troplanar}}
	\label{figure:graph_h}
\end{figure}

It follows from the above argument that most connected trivalent planar graphs (simple or otherwise) are sprawling.  Choosing different $H$'s can similarly show that most such graphs are crowded, and that most such graphs are TIE-fighters, providing alternate proofs that most connected trivalent planar graphs are not tropically planar.  We remark that no graphs resulting from this argument are $2$-edge-connected, and that it is unknown how the $2$-edge-connected counts $\mathscr{T}^{(2)}(g)$ and $\mathscr{P}^{(2)}(g)$ relate to one another.

The next result gives us an idea of how $\mathscr{P}(g)$ and $\mathscr{P}^{s}(g)$ grow with $g$.  Again, we have translated the results from vertices to genus; we have also translated from labelled to unlabelled graphs, using the fact that asymptotically $100\%$ of cubic graphs have no nontrivial symmetries.

\begin{theorem}[Theorems 1 and 4 from \cite{randomcubicfinal}]
Letting $\sim$ denote asymptotic equivalence, we have
$$ \mathscr{P}(g) \sim c(2g-2)^{-\frac{7}{2}} \gamma^{2g-2} $$
where $c\approx 0.104705$ and $\gamma\approx3.985537$; and we have
$$ \mathscr{P}^{\textrm{s}}(g) \sim d(2g-2)^{-\frac{7}{2}} \psi^{2g-2} $$
where $d\approx 0.030487$ and $\psi\approx3.132591$.
\label{theorem:planar_asymptotics}
\end{theorem}

Combined with Theorem \ref{theorem:most_not_troplanar}, this immediately implies the following.

\begin{cor}\label{cor:bound}
We have $\mathscr{T}(g)=o((2g-2)^{-7/2}\xi^g)$, where $\xi=\gamma^2\approx 15.88451$; and $\mathscr{T}^{\textrm{s}}(g)=o((2g-2)^{-7/2}\zeta^g)$, where $\zeta=\psi^2\approx 9.813126$.
\end{cor}

We spend the remainder of the section proving a stronger asymptotic bound on $\mathscr{T}(g)$.  Our main strategy is to stratify $\mathscr{T}(g)$ by separating out the graphs coming from polygons of different lattice widths.  Let $\mathscr{T}^{\textrm{lw}}_\ell(g)$ denote the number of troplanar graphs arising from polygons of genus $g$ with lattice width exactly $\ell$, and let $\mathscr{T}^{\textrm{lw}}_{\geq \ell}(g)$ denote the number of troplanar graphs arising from polygons of genus $g$ with lattice width at least $\ell$. Since all polygons of genus $g\geq 2$ have lattice width at least $2$, we have $\mathscr{T}(g)\leq \mathscr{T}^{\textrm{lw}}_{2}(g)+\mathscr{T}^{\textrm{lw}}_{3}(g)+\mathscr{T}^{\textrm{lw}}_{\geq 4}(g)$, where the inequality comes from the fact one graph may come from multiple polygons.  We will separately bound each of these three summands, thus giving an overall bound on $\mathscr{T}(g)$.

The easiest term to handle is $\mathscr{T}^{\textrm{lw}}_2(g)$.  If a polygon has lattice width $2$ and genus at least $2$, it must have a one-dimensional interior polygon and thus be hyperelliptic.  Due to the work of \cite{chan}, \cite[\S 5]{BJMS}, and \cite{morrisonhyp}, the case for hyperelliptic polygons is understood quite well: for each genus $g$ there are $2^{g-2}+2^{\lfloor (g-2)/2\rfloor}$ distinct skeletons that arise from some hyperelliptic polygon of genus $g$; see Section \ref{sec:lower} for more details.  Thus we have an explicit formula for $\mathscr{T}^{\textrm{lw}}_2(g)$.

We now move on to bounding $\mathscr{T}^{\textrm{lw}}_{3}(g)$.  Our starting point for this is the following result, which considers a single polygon of lattice width $3$.

\begin{prop}  A nonhyperelliptic polygon of lattice width $3$ gives rise to $O(8^g/\sqrt{g})$ troplanar graphs.
\label{prop:width_3}
\end{prop}

\begin{proof}
Let $\Delta$ be a nonhyperelliptic polygon of lattice width $3$.   First we note that $\Delta$ cannot be a standard triangle, as the only standard triangle of lattice width $3$  has genus $1$ and is thus hyperelliptic. It follows from \cite[Theorem 4]{castryckcools} that $\textrm{lw}(\Delta^{(1)})=\textrm{lw}(\Delta)-2=1$.  This means that $\Delta^{(1)}$ has no interior lattice points, which combined with its lattice width implies that up to equivalence, $\Delta^{(1)}$ must be a right trapezoid of height one, with vertices at $(0,0)$, $(0,1)$, $(a,0)$, and $(1,b)$, where $a\geq b\geq 1$ and $a+b+2=g$ \cite{Koelman, Castryck2012}.

To show that $\Delta$ gives rise to $O(8^g/\sqrt{g})$ graphs, we will first show that it gives rise to $O(4^g/\sqrt{g})$ $2$-edge-connected troplanar graphs of genus $g$.  A key fact is that if a triangulation $\mathcal{T}$ gives rise to a a $2$-edge-connected troplanar graph $G$, then $G$ is determined by a very small amount of information in $\mathcal{T}$, namely by the edges connecting the interior lattice points of $\Delta$.  Every interior lattice point of $\Delta$ in must contribute a bounded face to the troplanar graph; since there are no bridges, for each pair of the $g$ bounded faces of the graph, they either share a common edge or are non-adjacent. Since two bounded faces share an edge if and only an edge connects the corresponding interior points in the triangulation, we can draw the 2-edge-connected troplanar graph with only this limited information.

Let $\mathcal{T}$ be a regular unimodular triangulation of $\Delta$ giving rise to a $2$-edge-connected troplanar graph $G$.  Let $\mathcal{T}'$ consist of all edges in $\mathcal{T}$ with both endpoints at interior lattice points of $\Delta$. Letting $\partial\left(\Delta^{(1)}\right)$ denote the boundary of $\Delta^{(1)}$, we  claim that $\mathcal{T}'\cup \partial\left(\Delta^{(1)}\right)$ is a unimodular triangulation of $\Delta^{(1)}$, as illustrated in Figure \ref{figure:width3}. Certainly it is a subdivison: any choice of non-crossing edges between the top and bottom rows of a polygon of lattice width $1$ gives a subdivision.

\begin{figure}[hbt]
   		 \centering
		\includegraphics[scale=1]{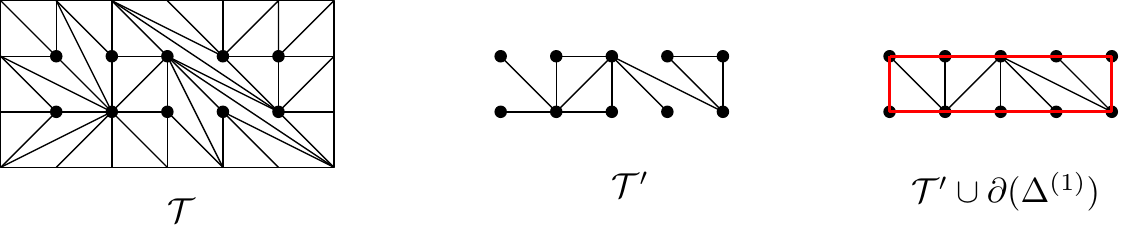}
	\caption{A triangulation $\mathcal{T}$ of a lattice width $3$ polygon $\Delta$; the subset $\mathcal{T}'$; and the resulting triangulation $\mathcal{T}'\cup \partial\left(\Delta^{(1)}\right)$ of $\Delta^{(1)}$}
	\label{figure:width3}
\end{figure}

Suppose for the sake of contradiction that $\mathcal{T}'\cup \partial\left(\Delta^{(1)}\right)$ is not a unimodular triangulation.  Then there is a subpolygon $\Sigma$ of $\Delta^{(1)}$ present in $\mathcal{T}'\cup \partial\left(\Delta^{(1)}\right)$ that is not a unimodular triangle.  
The polygon $\Sigma$ must have two adjacent boundary points $p$ and $q$ such that the edge $pq$ is not present in $\mathcal{T}$; otherwise $\Sigma$ would be a closed face in $\mathcal{T}$, and $\mathcal{T}$ would not be a unimodular triangulation. Note that all edges in $\partial(\Sigma)\setminus \partial(\Delta^{(1)})$ are present in $\mathcal{T}$, so $pq\in\partial(\Delta^{(1)})$.  If $p$ and $q$ are in different rows, then either $\{p,q\}=\{(0,0),(0,1)\}$ or $\{p,q\}=\{(a,0),(b,1)\}$. In the first case, $\mathcal{T}$ must have either an edge from $(1,0)$ to $(-1,1)$ or from $(1,1)$ to $(-1,0)$, which then yields either an edge from $(1,0)$ to $(0,1)$ or from $(1,1)$ to $(0,0)$; this means that $\Sigma$ is a unimodular triangle, a contradiction.  A similar contradiction arises if $\{p,q\}=\{(a,0),(b,1)\}$. Thus $p$ and $q$ must be in the same row.

Because $pq$ is not present in $\mathcal{T}$, some $e$ in $\mathcal{T}\setminus\mathcal{T}'$ must pass through $pq$.  It cannot be a nontrivial split, since $G$ is $2$-edge-connected, so $e$ must have endpoints $r$ and $s$ where $r$ is an interior lattice point and $s$ is a boundary lattice point, as illustrated in Figure \ref{figure:must_be_split}. Note that $r$ must be on the opposite row from $p$ and $q$. At this point no edge in $\mathcal{T}$ can separate $p$ from $r$:  such an edge would have to cross $e$.  Thus $pr$ is an edge in $\mathcal{T}$.  Similarly, $qr$ is an edge in $\mathcal{T}$.  It follows that the triangle $pqr$ appears in the subdivision $\mathcal{T}'\cup\partial(\Delta^{(1)})$ of $\Delta^{(1)}$.  By Pick's Theorem, $pqr$ is unimodular.  Since $p$ and $q$ are on the same row, they can only be common vertices of one polygon in the subdivision of $\Delta^{(1)}$, so $\Sigma$ must be this unimodular triangle, a contradiction.  We conclude that $\mathcal{T}'\cup \partial\left(\Delta^{(1)}\right)$ is a unimodular triangulation.

\begin{figure}[hbt]
   		 \centering
		\includegraphics[scale=1]{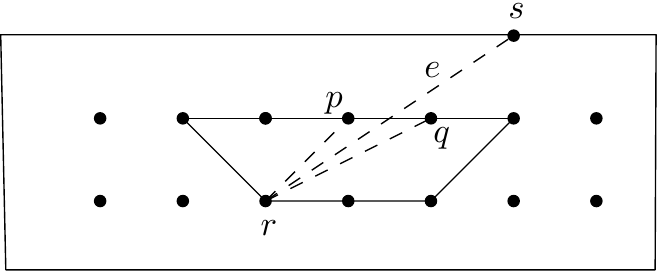}
	\caption{The relative positions of $p$, $q$, $r$, and $s$.  The polygon $\Sigma$ is solid, and necessary edges in $\mathcal{T}$ are dashed.}
	\label{figure:must_be_split}
\end{figure}

Ignoring for a moment whether $(0,0)$ connects to $(0,1)$ and whether $(a,0)$ connects to $(b,1)$ in $\mathcal{T}$, this means that the number of ways the upper and lower lattice points of $\Delta^{(1)}$ can connect to one another in $\mathcal{T}$ is equal to the number of unimodular triangulations of $\Delta^{(1)}$. 
By \cite[\S 2]{countinglattice}, $\Delta^{(1)}$ has $\binom{a+b}{a}$ unimodular triangulations.  
We have $a+b=g-2$, so $\binom{a+b}{a}=\binom{g-2}{a}\leq\binom{g-2}{\lfloor(g-2)/2\rfloor}$.  
This central binomial coefficient is asymptotic to $\frac{2^{g-2}}{\sqrt{\pi(g-2)/2}}$.

The only remaining data regarding the connections between the interior lattice points is which edges from the boundary of $\Delta^{(1)}$ are present in the triangulation $\mathcal{T}$.  Since $\Delta^{(1)}$ has $g$ boundary points, it also has $g$ boundary edges, so there are $2^{g}$ ways to choose which are included and which are not.  Multiplying this by our binomial coefficient bound, we find that the number of ways the interior lattice points of $\Delta$ could be connected in order to yield a $2$-edge-connected troplanar graph is $O(4^g/\sqrt{g})$.

By the same argument from Corollary \ref{corollary:two-connected}, the total number of troplanar graphs arising from $\Delta$ is at most $2^{g-1}$ times the number of $2$-edge-connected graphs arising from $\Delta$.  We conclude the number of troplanar graphs arising from $\Delta$ is $O(8^g/\sqrt{g})$.
\end{proof}

\begin{cor}\label{corollary:width_3}  We have $\mathscr{T}^{\textrm{lw}}_3(g)=O(8^g\cdot \sqrt{g})$.
\end{cor}

\begin{proof}
As already noted, every nonhyperelliptic  polygon of lattice width $3$ and genus $g$ has an interior polygon equal to a right trapezoid of height $1$ with $g$ lattice points.  There are $O(g)$ such possible interior polygons, and each has at most one associated maximal polygon.  Every one of these maximal polygons contributes  $O(8^g/\sqrt{g})$ troplanar graphs, so in total we have $\mathscr{T}^{\textrm{lw}}_3(g)=O(8^g\cdot \sqrt{g})$.
\end{proof}

We now need to bound $\mathscr{T}^{\textrm{lw}}_{\geq 4}(g)$.  It turns out that our argument will hold in general for $\mathscr{T}^{\textrm{lw}}_{\geq \ell+1}(g)$ for any $\ell\geq 3$, so we work in that generality before specializing to $\ell=3$ to obtain our final result.  We will find a bound for the number of all lattice polygons of genus $g$; then we will find a bound for how many unimodular triangulations a polygon of lattice width $\ell+1$ (or more) can have.  Multiplying these together will give an upper bound on $\mathscr{T}^{\textrm{lw}}_{\geq \ell+1}(g)$.

For bounding the number of polygons, our starting point is the following result, which bounds the number of polygons with at most a fixed area. The upper bound comes from \cite{bv}, the lower bound from \cite{arnold}.

\begin{theorem}\label{theorem:area}
Let $A$ be a fixed rational number, and let $N(A)$ be the number of convex lattice polygons distinct under lattice transformations with area less than $A$. Then there exist two positive constants $C_1$, $C_2$ such that
$$ C_1A^{1/3} \leq \log( N(A)) \leq C_2A^{1/3}$$
\end{theorem}

To convert this result from bounded area to  bounded genus, we need to determine how many lattice points are possible in a polygon of genus $g$, and then apply Theorem \ref{theorem:picks}. We will use the following fact from derived in \cite{Castryck2012}, presented as Equation (2) following the proof of their Theorem 1.

\begin{prop}\label{nonhyp}
Let $\Delta$ be a nonhyperelliptic lattice polygon with $v$ vertices and with $r$ lattice boundary points, and let $r^{(1)}$ be the number of lattice boundary boundary points of $\Delta^{(1)}$. Then

$$r\leq r^{(1)}+12-v.$$
\end{prop}

Since $v\geq 3$ and $r^{(1)}\leq g$, we immediately obtain the following corollary.

\begin{cor}\label{cor:bound_on_r}
If $\Delta$ is nonhyperelliptic with $r$ lattice boundary points and $g$ interior lattice points, then $r\leq g+9$.
\end{cor}

By applying Pick's Theorem and this corollary, we can then use Theorem \ref{theorem:area} to obtain an upper bound for the number of lattice polygons of genus $g$.

\begin{lemma}\label{lemma:polygoncount}
There is a positive constant $C_2$ such that the number of convex nonhyperelliptic lattice polygons of a genus $g$ is at most $e^{C_2(\frac{3g}{2}+9)^{1/3}}$.
\end{lemma}

\begin{proof}
If $\Delta$ is a  nonhyperelliptic  polygon of genus $g$, it has at most $g+9$ lattice boundary points by Corollary \ref{cor:bound_on_r}.  By Pick's Theorem, the area of $\Delta$ is $r+\frac{g}{2}-1$, which is at most $g+9+\frac{g}{2}-1=\frac{3g}{2}+8$.  This means that in order to bound the number of nonhyperelliptic  polygons of genus $g$, it suffices to bound the number of polygons with area less than $\frac{3g}{2}+8+1=\frac{3g}{2}+9$.  In the notation of Theorem \ref{theorem:area}, we are bounding $N\left(\frac{3g}{2}+9\right)$.  That theorem tells us there exists a positive constant $C_2$ such that that 
\[\log(N\left(\frac{3g}{2}+9\right))\leq C_2(\frac{3g}{2}+9)^{1/3}, \]
which can be rewritten as
\[N\left(\frac{3g}{2}+9\right)\leq e^{C_2(\frac{3g}{2}+9)^{1/3}}.  \]
Thus the number of nonhyperelliptic polygons of genus $g$ is bounded by $e^{C_2(\frac{3g}{2}+9)^{1/3}}$
\end{proof}

We now need to bound the number of unimodular triangulations a polygon of lattice width at least $\ell+1$ can have.  To do this, we prove the following proposition, which can be viewed as a stronger version of Corollary \ref{cor:bound_on_r} for this class of polygons.

\begin{prop}
\label{prop:big_lw}
Suppose that $\Delta$ is a polygon with of lattice width $\textrm{lw}(\Delta)\geq \ell+1$ where $\ell\geq 3$, and let $\Delta$ have genus $g$ and $r$ lattice boundary points.  Then $r\leq \frac{2g}{\ell}+4\sqrt{g+8/3}+2$
\end{prop}

\begin{proof}
First we recall the bound $\textrm{lw}(\Delta)^2\leq 8\textrm{Vol}(\Delta)/3$, proven in \cite{thinnest} and also presented in \cite[Lemma 5.2(vi)]{castryckpencil}.  By Pick's theorem we know $\textrm{Vol}(\Delta)=g+\frac{r}{2}-1$, and since $\Delta$ cannot be hyperelliptic due to its lattice width we know $r\leq g+9$ by Corollary \ref{cor:bound_on_r}. Combining these we find
\[\textrm{lw}(\Delta)^2\leq 8\textrm{Vol}/3=\frac{8}{3}\left(g+\frac{r}{2}-1\right)\leq \frac{8}{3}\left(g+\frac{g+9}{2}-1\right)=4g+\frac{32}{3}.\]
Square-rooting both sides yields $\textrm{lw}(\Delta)\leq \sqrt{4g+32/3}=2\sqrt{g+8/3}$. We now use \cite[Theorem 8]{Castryck2012}, which states
\[r\leq \frac{2}{\textrm{lw}(\Delta)-1}\cdot g+2\cdot \left(\textrm{lw}(\Delta)+1\right).\]  Since $\textrm{lw}(\Delta)\geq \ell+1$ and $\textrm{lw}(\Delta)\leq 2\sqrt{g+8/3}$, we have
\[r\leq \frac{2}{(\ell+1)-1}\cdot g+2\cdot\left(2\sqrt{g+8/3}+1\right)=\frac{2g}{\ell}+4\sqrt{g+8/3}+2,\]
as desired.
\end{proof}

\begin{prop}\label{prop:width_4}
A polygon of genus $g$ and lattice width at least $\ell+1$ admits at most 
$2^{\left(3+\frac{2}{\ell}\right)g+4\sqrt{g+8/3}-1}$ unimodular triangulations
\end{prop}

\begin{proof}
Let $\Delta$ be a nonhyperelliptic polygon of genus $g$ with $r$ lattice boundary points and lattice width at least $\ell+1$.   By  \cite[Theorem 9.3.7]{triangulations}, $\Delta$ admits at most  $2^{3g+r-3}$ unimodular triangulations.  By Proposition \ref{prop:big_lw}
 we have $r\leq \frac{2g}{\ell}+4\sqrt{g+8/3}+2$.  It follows that the number of unimodular triangulations of $\Delta$ is bounded by \[2^{3g+\left(\frac{2g}{\ell}+4\sqrt{g+8/3}+2\right)-3}=2^{\left(3+\frac{2}{\ell}\right)g+4\sqrt{g+8/3}-1}.\]
\end{proof}

Combined with our bound on the number of lattice polygons of genus $g$, this gives the following result.

\begin{cor}\label{corollary:width_4}
We have $\mathscr{T}^{\textrm{lw}}_{\geq \ell+1}(g)\leq e^{C_2(\frac{3g}{2}+9)^{1/3}}\cdot2^{\left(3+\frac{2}{\ell}\right)g+4\sqrt{g+8/3}-1}$, where $C_2$ is the constant from Lemma \ref{lemma:polygoncount}.
\end{cor}

We can now prove the following upper bound on $\mathscr{T}(g)$.

\begin{theorem}\label{theorem:best}
We have
\[\mathscr{T}(g)=O\left(2^{\frac{11g}{3}+O(\sqrt{g})}\right).\]
\end{theorem}

\begin{proof}
As already noted, we have
\[\mathscr{T}(g)\leq \mathscr{T}^{\textrm{lw}}_{2}(g)+\mathscr{T}^{\textrm{lw}}_{3}(g)+\mathscr{T}^{\textrm{lw}}_{\geq 4}(g)\]
We know $\mathscr{T}^{\textrm{lw}}_{2}(g)=O(2^g)$, and by Corollary \ref{corollary:width_3} we have $\mathscr{T}^{\textrm{lw}}_{2}(g)=O(8^g/\sqrt{g})$.  It follows from Corollary \ref{corollary:width_4} when $\ell=3$ that $\mathscr{T}^{\textrm{lw}}_{\geq 4}(g)=O\left(2^{\frac{11g}{3}+O(\sqrt{g})}\right)$.  The largest of the three bounds is the one coming from $\mathscr{T}^{\textrm{lw}}_{\geq 4}(g)$, and so this serves as our bound for $\mathscr{T}(g)$.
\end{proof}

It follows that $\mathscr{T}(g)=o\left((2^{11/3}+\varepsilon)^{g}\right)$ for any $\varepsilon>0$, where $2^{11/3}\approx12.699$.  This illustrates that Theorem \ref{theorem:best} is indeed a stronger bound on $\mathscr{T}(g)$ than the result from Corollary \ref{cor:bound}.  It also provides another argument that most connected trivalent planar graphs are not tropical: by Theorem \ref{theorem:planar_asymptotics}, $\mathscr{P}(g)=\Omega\left((2g-2)^{-7/2}\xi^g\right)$ where $\xi\approx 15.88451$.  Due to this exponential base being larger than $2^{11/3}+\varepsilon$, the ratio $\mathscr{T}(g)/\mathscr{P}(g)$ rapidly goes to $0$ as $g$ increases.

We close this section by discussing a few ways in which our upper bound might be improved.   One strategy could be to stratify further by lattice width.  Fix $\ell\geq 3$, and consider the bound
\[\mathscr{T}(g)\leq \mathscr{T}^{\textrm{lw}}_{2}(g)+\cdots+\mathscr{T}^{\textrm{lw}}_{\ell}(g)+\mathscr{T}^{\textrm{lw}}_{\geq \ell+1 }(g).\]
We know by Corollary \ref{corollary:width_4} that 
\[\mathscr{T}^{\textrm{lw}}_{\geq\ell+1}(g)=O\left(2^{\left(3+\frac{2}{\ell}\right)g+O(\sqrt{g})}\right).\]
If we can effectively bound $\mathscr{T}^{\textrm{lw}}_{i}(g)$ for $i\leq\ell$, this could lead to an improvement on Theorem \ref{theorem:best}.  So far we have bounds on $\mathscr{T}^{\textrm{lw}}_{2}(g)$ and $\mathscr{T}^{\textrm{lw}}_{3}(g)$, so $\mathscr{T}^{\textrm{lw}}_{4}(g)$ would be the next step. We remark that this strategy will never lead to a stronger bound than $O(2^{(3+\varepsilon)g})$ on  $\mathscr{T}(g)$.

Another possible direction for future improvement comes from the observation that the bound from Proposition \ref{prop:width_4} is a bound on \emph{all} unimodular triangulations, when we only need to consider those unimodular triangulations that are \emph{regular}. Numerics such as those  computed in \cite{countinglattice} suggest that regular triangulations are much rarer than non-regular triangulations for large polygons.  A precise enough result to this effect could lower the number of triangulations we consider.

Even restricting to regular triangulations, we note that many triangulations can give the same graph, even though our bound from Corollary \ref{corollary:width_4} comes from bounding the total number of triangulations.  For instance, as computed in \cite{BJMS}, the unique maximal nonhyperelliptic polygon of genus $3$ admits $1278$ regular unimodular triangulations up to symmetry, but only yields four distinct troplanar graphs.  Unfortunately, we cannot in general say that each troplanar graph arises from many different triangulations:  the graph of genus $4$ illustrated in Figure \ref{figure:minor} only arises from a single triangulation of a single polygon, as shown in \cite{BJMS}.

We also remark on one way in which our argument is already essentially optimal. Although it is a less dramatic contributor to our bound, we might wonder if we are significantly overestimating the number of polygons of genus $g$ that we must consider.  The reader may notice that the bound from Lemma \ref{lemma:polygoncount} is for general convex lattice polygons, while we only need to work with maximal polygons. The following arguments show that in fact this will not meaningfully change our upper bound from Lemma \ref{lemma:polygoncount}. 

Note that if a polygon is nonmaximal, then at least one of its edges has lattice length $1$.  This is because any nonmaximal polygon can be constructed by taking a maximal polygon with the same interior polygon, removing carefully chosen boundary points, and taking the convex hull of the remaining points; assuming the interior polygon has not been changed, this always produces an edge of lattice length $1$. Contrapositively, if all sides of a polygon have lattice length $2$ or more, then that polygon is maximal. Let $\Delta$ be a lattice polygon, and let $\Delta '$ be $2\Delta$, the polygon obtained by scaling $\Delta$ by a factor of $2$.  In going from $\Delta$ to $\Delta'$, the perimeter increases by a factor of $2$, and the area increases by a factor of $4$.

\begin{lemma}  Let $\Delta$ and $\Delta '$ have $g$ and $g'$ interior lattice points and $r$ and $r'$ boundary lattice points, respectively.  Then $g'=4g+r-3$.  Moreover, if $g\geq 1$, then $g'\leq 6g+4$.
\end{lemma}

\begin{proof} By Pick's theorem, the area of $\Delta$ is $\frac{r}{2}+g-1$ and the area of $\Delta'$ is $\frac{r'}{2}+g'-1$,  We know that the area of $\Delta'$ is four times that of $\Delta$, so $2r+4g-4=\frac{r'}{2}+g'-1$. We also know that $r'=2r$, since the number of boundary points is equal to the (lattice) perimeter, giving us $2r+4g-4=r+g'-1$.  Solving for $g'$ gives $g'=4g+r-3$.

For the inequality on $g'$, we will use the fact that $r\leq 2g+7$ for any lattice polygon of genus at least $1$  \cite{scott}.  Since   $g'=4g+r-3$, we have $g'\leq 4g+(2g+7)-3=6g+4$, as claimed.
\end{proof}

\begin{prop}  Let $g\geq 1$.  If there are $N$ lattice polygons of positive genus at most $g$, then there are at least $N$ maximal lattice polygons of genus at most $6g+4$.
\end{prop}

\begin{proof}  Assume there are $N$ (distinct) lattice polygons of positive genus at most $g$.  Mapping each polygon $\Delta$ to $2\Delta$ gives a collection of $N$ distinct maximal lattice polygons, since each side of $2\Delta$ has length at least $2$.  By the previous lemma, each of these polygons has genus at most $6g+4$, proving our claim.
\end{proof}

We now use the lower bound from Theorem \ref{theorem:area}, which says that there is a positive constant $C_1$ such that $C_1A^{1/3}\leq \log(N(A))$, where $N(A)$ is the number of convex lattice polygons with area less than $A\in\mathbb{Q}^+$.  Note that a lattice polygon of genus $g$ has area at least $g+\frac{3}{2}-1=g+\frac{1}{2}$ by Pick's Theorem.  This means that the number of polygons of genus at most $g$ is an upper bound for $N(g+1)$, the number of polygons with area less than $g+1$.  Thus any lower bound on $N(g+1)$ is also a lower bound on the number of polygons of genus at most $g$.  By a similar argument to that presented in Lemma \ref{lemma:polygoncount}, we have $e^{C_1(g+1)^{1/3}}\leq N(g+1)$, so there are at least $e^{C_1(g+1)^{1/3}}$ polygons of genus at most $g$. Only a few of these polygons can have genus $0$: by the classification result in \cite{Koelman,Castryck2012}, only quadratically many genus $0$ polygons have a fixed area $A$, so only cubically many genus $0$ polygons have area at most $A$.  This is dwarfed by the exponential in $g^{1/3}$, so we can replace $C_1$ with another positive constant $C_1'$ to conclude that there are at least $e^{C_1'(g+1)^{1/3}}$ polygons of positive genus at most $g$ (assuming that $g\geq 1$). 

Solving $g'\leq 6g+4$ for $g$ gives $g\geq \frac{g'-4}{6}$, yielding the following bound.

\begin{cor}  There exists a positive constant $C_1'$ such that for all $g\geq 10$, there are at least $e^{C_1'\left((g+2)/6\right)^{1/3}}$ maximal polygons of genus at most $g$.
\end{cor}

\begin{proof}
Our assumption on $g$ guarantees that $(g-4)/6\geq 1$.  There are at least $e^{C_1'(((g-4)/6)+1)^{1/3}}=e^{C_1'\left((g+2)/6\right)^{1/3}}$ lattice polygons of positive genus at most $(g-4)/6$.  For each such polygon $\Delta$, the maximal polygon $2\Delta$ has genus at most $g$.  It follows that there are at least $e^{C_1'\left((g+2)/6\right)^{1/3}}$ maximal polygons of genus at most $g$. 
\end{proof}

\section{A lower bound on the number of troplanar graphs}
\label{sec:lower}

As mentioned briefly in the previous section, hyperelliptic polygons of genus $g$ give rise to  $2^{g-2}+2^{\lfloor (g-2)/2\rfloor}$ distinct skeletons.  To see where this formula comes from, we recall the following argument from \cite[\S 6]{BJMS}. All graphs that arise from hyperelliptic polygons are \emph{chains}, meaning that they have the structure of a sequential chain of loops with two adjacent chains either sharing a common edge or being joined by a bridge.  The three chains of genus $3$ are illustrated in Figure \ref{figure:chains}, along with subdivisions of a hyperelliptic polygon that give rise to them.  One can specify a chain by a binary string of length $g-1$, describing what the sequence of bridges and shared edges is; for the illustrated graphs, these strings would be $00$, $01$, and $11$, where $0$ is a shared edge and $1$ is a bridge. Two strings give the same graph if and only if they are reverses of one another, like $01$ and $10$.  The number of binary strings of length $g-1$, up to this reversing equivalence, is $2^{g-2}+2^{\lfloor (g-2)/2\rfloor}$. This gives us a lower bound on $\mathscr{T}(g)$, and we may write $\mathscr{T}(g)=\Omega(2^g)$.

\begin{figure}[hbt]
   		 \centering
		\includegraphics[scale=1]{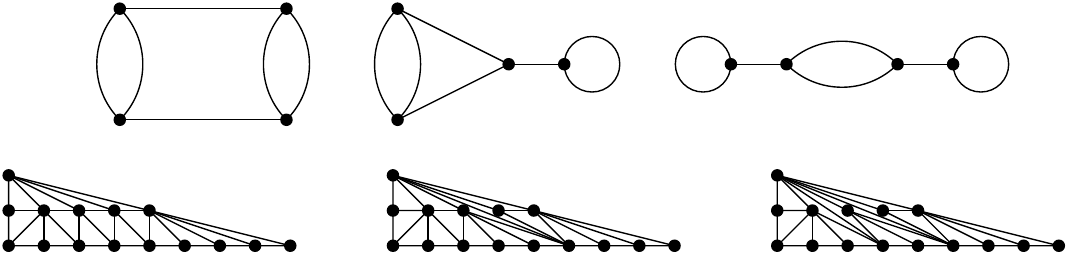}
	\caption{The three chains of genus three, and triangulations of a hyperelliptic polygon giving rise to them}
	\label{figure:chains}
\end{figure}

Our goal for this section is to provide a  better lower bound on $\mathscr{T}(g)$ by constructing a family of graphs that is asymptotically larger than the family of chains.  To do so, we will construct many unimodular triangulations of a polygon of genus $g$, argue that these triangulations are regular, and prove that each triangulations gives rise to different troplanar graph. These are in general challenging endeavors which are much simpler in the hyperelliptic case. First, all unimodular triangulations of hyperelliptic polygons are regular \cite[Proposition 3.4]{countinglattice}.  Second, determining whether two chains are isomorphic is relatively simple: they must have the same pattern of shared edges and bridges, up to perhaps flipping the graph around.  Our family of troplanar graphs will be constructed so that each pair of graphs can be similarly verified to be non-isomorphic.  We begin with the following result.

\begin{prop}\label{prop:isomorphic} Suppose we have two trivalent graphs $G$ and $H$ of genus $g$ that are of the following forms:
$$\includegraphics{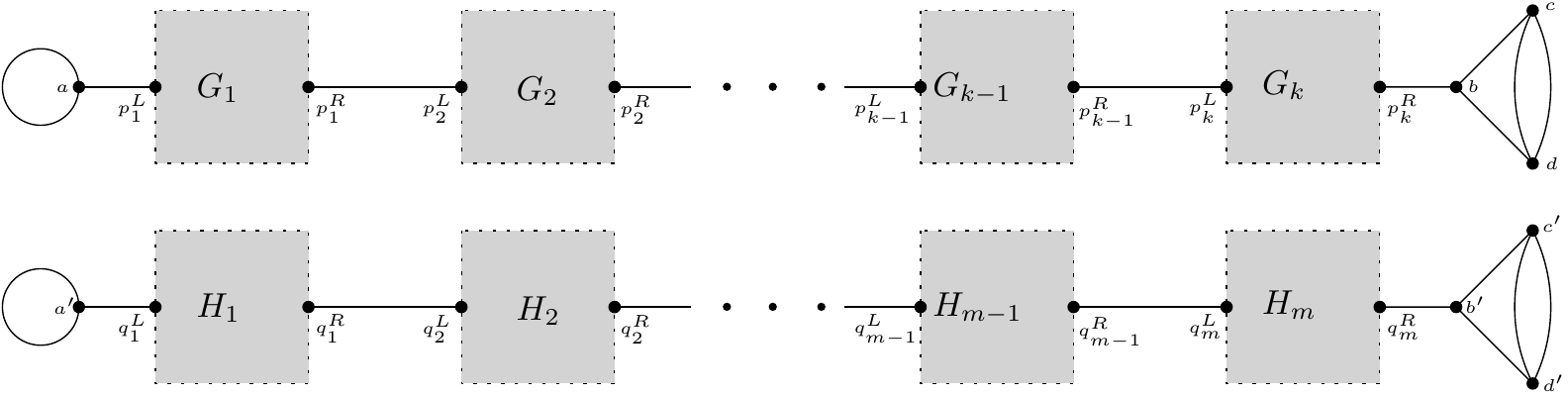}$$
where each box contains a positive-genus $2$-edge-connected component $G_i$ of $G$ or $H_i$ of $H$, respectively, with exactly two $2$-valent vertices where the bridges connect; call these vertices $p_i^{L}$ and $p_i^R$ for $G_i$, and $q_i^{L}$ and $q_i^R$ for $H_i$.  Then $G$ is isomorphic to $H$ if and only if $k=m$ and for each $i$, $G_i$ is isomorphic to $H_i$ via an isomorphism sending $p_i^{L}$ to $q_i^{L}$ and $p_i^{R}$ to $q_i^{R}$.

\end{prop}

\begin{proof} 

If $k=m$ and $G_i$ is isomorphic to $H_i$ for all $i$ via an isomorphism sending $p_i^L$ to $q_i^L$ and $p_i^R$ to $q_i^R$, then there is a natural isomorphism between $G$ and $H$:  simply use the isomorphisms on the $2$-edge-connected components, and send $a,b,c,$ and $d$ to $a',b',c',$ and $d'$, respectively.

Now assume $G$ is isomorphic to $H$.  Since $G$ has $k+1$ bridges and $H$ has $m+1$ bridges, we must have $k=m$ since isomorphism preserves the number of bridges.  Let $\varphi:V(G)\rightarrow V(H)$ be a graph isomorphism.  Since $a$ and $a'$ are the only vertices incident to a loop, we must have $\varphi(a)=a'$.  Since $a$ is only adjacent to $p_1^L$ and $a'$ is only adjacent to $q_1^L$, we must have $\varphi(p_1^L)=q_1^L$.  Any isomorphism of graphs will map $2$-edge-connected components to $2$-edge-connected components, and since one vertex in $G_1$ is mapped to one vertex in $H_1$, $\varphi$ restricted to $V(G_1)$ must give an isomorphism from $G_1$ to $H_1$.  Only two vertices in $G_1$ are incident to a bridge in $G$, namely $p_1^L$ and $p_1^R$; the same is true for $H_1$ in $H$, namely $q_1^L$ and $q_1^R$.  Since $\varphi(p_1^L)=q_1^L$, we must have $\varphi(p_1^R)=q_1^R$.  Thus, $G_1$ and $H_1$ are isomorphic via an isomorphism that maps $p_1^{L}$ to $q_1^{L}$ and $p_1^{R}$ to $q_1^{R}$.  Since all vertices incident to $p_1^R$ have been accounted for except $p_2^L$, and since the same holds for $q_1^R$  except for $q_2^L$, we must have $\varphi(p_2^L)=q_2^L$.  We may then apply an identical argument to establish the desired isomorphism from $G_2$ to $H_2$, and so on, all the way up to the fact that $p_k^R$ must be sent to $q_k^R$. This completes the proof. 
\end{proof}

We will construct a number of regular triangulations which give rise to graphs of the form considered in Proposition \ref{prop:isomorphic}.   For even $g$, let $P^{||}_g$ denote the parallelogram of genus $g$ with vertices at $(0,3)$, $(1,0)$, $(g/2,3)$, and $((g+2)/2,0)$, as pictured in Figure \ref{figure:parallelogram}.  We will construct regular triangulations of $P^{||}_g$ by tiling it with triangulated copies of $P^{||}_2$, $P^{||}_4$, and $P^{||}_6$.  We will then slightly modify the polygon so that the troplanar graphs obtained from our triangulations have the forms prescribed by Proposition \ref{prop:isomorphic}.

\begin{figure}[hbt]
   		 \centering
		\includegraphics[scale=1]{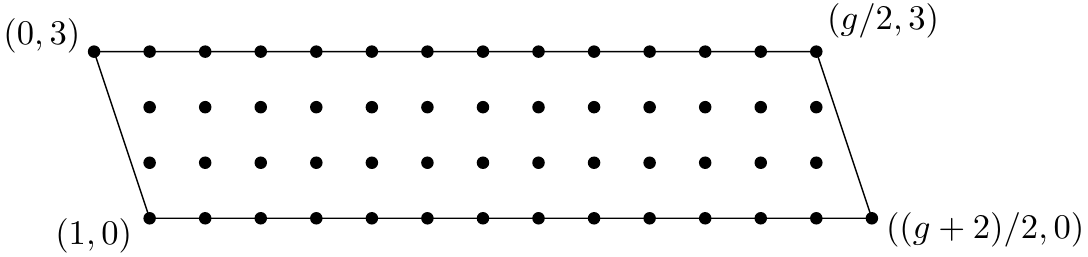}
	\caption{The parallelogram $P^{||}_g$}
	\label{figure:parallelogram}
\end{figure}

We will refer to the  triangulated copies of  $P^{||}_2$ (respectively of $P^{||}_4$ and $P^{||}_6$) as \emph{tiles of genus 2} (respectively \emph{tiles of genus $4$} and \emph{tiles of genus $6$}). To start out, we will use only one tile of genus $2$, namely the one illustrated in Figure \ref{figure:genus2_tile}, which also pictures  a dual tropical curve and the corresponding troplanar graph.  This is not the only troplanar graph that can be obtained from $P^{||}_2$; however, it is the only bridge-less one, which is important if we want to apply Proposition \ref{prop:isomorphic}.  We have also marked two points on the graph as $L$ and $R$:  these are the points where the bridges would attach to this $2$-edge-connected component if the tile appeared in the middle of a triangulation of a larger  $P^{||}_g$.

\begin{figure}[hbt]
   		 \centering
		\includegraphics[scale=1]{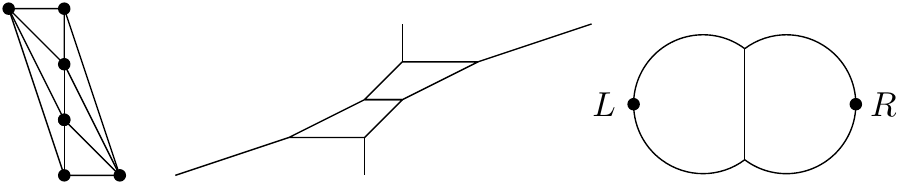}
	\caption{The tile of genus $2$, a dual tropical curve, and the troplanar graph with two marked points}
	\label{figure:genus2_tile}
\end{figure}

We will start out with $8$ tiles of genus $4$, as illustrated in Figure \ref{figure:genus4_tiles}.  Although the boxed groups of graphs would be isomorphic without the marked points, the marked points make them distinct.  We will also begin with $49$ tiles of genus $6$, presented in Appendix \ref{appendix}.  It is important to verify that all the tiles we use are regular triangulations.  Using TOPCOM \cite{TOPCOM}, we searched for all non-regular triangulations of $P^{||}_2$, $P^{||}_4$, and $P^{||}_6$; it turns out these polygons have no non-regular triangulations, so all our tiles are safe to use.

\begin{figure}[hbt]
   		 \centering
		\includegraphics[scale=0.9]{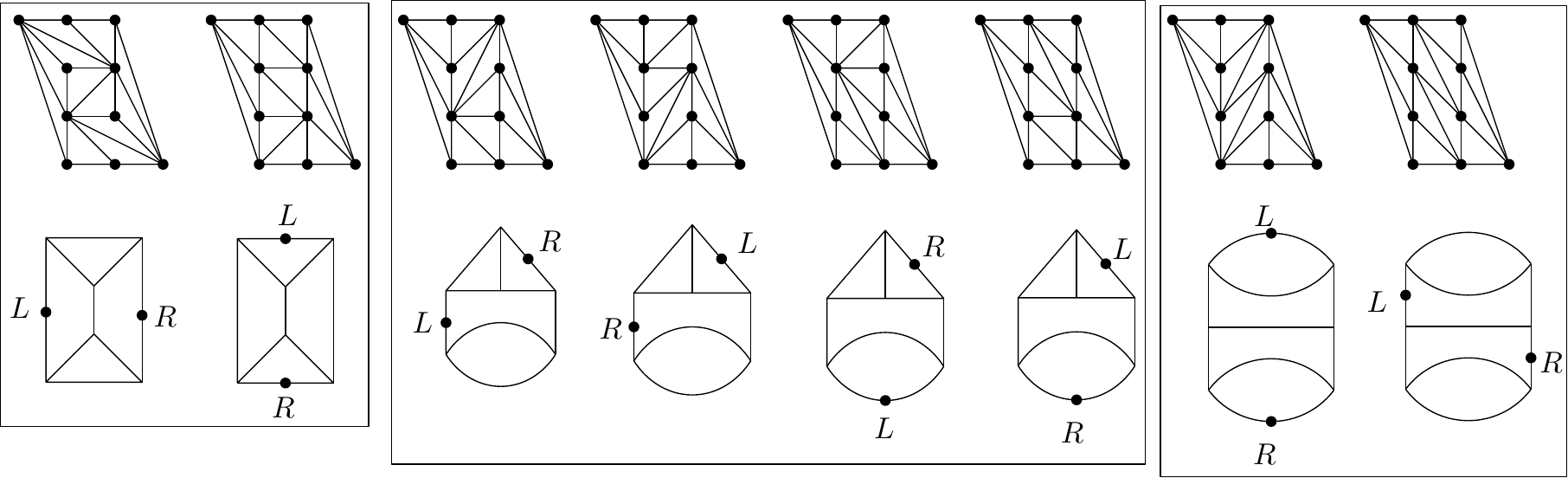}
	\caption{The first eight tiles of genus $4$}
	\label{figure:genus4_tiles}
\end{figure}

An example of a tiling of $P^{||}_{14}$ is presented in the top left Figure \ref{figure:tiled}.  This is a regular triangulation, since patching regular triangulations along edges of lattice length $1$ preserves regularity, as noted in \cite[Proposition 3.4]{countinglattice}. The troplanar graph obtained is also illustrated.  Although it is not quite of the form prescribed by Proposition \ref{prop:isomorphic}, a small modification of the polygon can fix this.  In particular, we can expand our polygon to the quadrilateral of genus $17$ pictured in the bottom left of Figure \ref{figure:tiled}, and triangulating it as shown adds on the left loop and the right genus-$2$ structure.

\begin{figure}[hbt]
   		 \centering
		\includegraphics[scale=0.9]{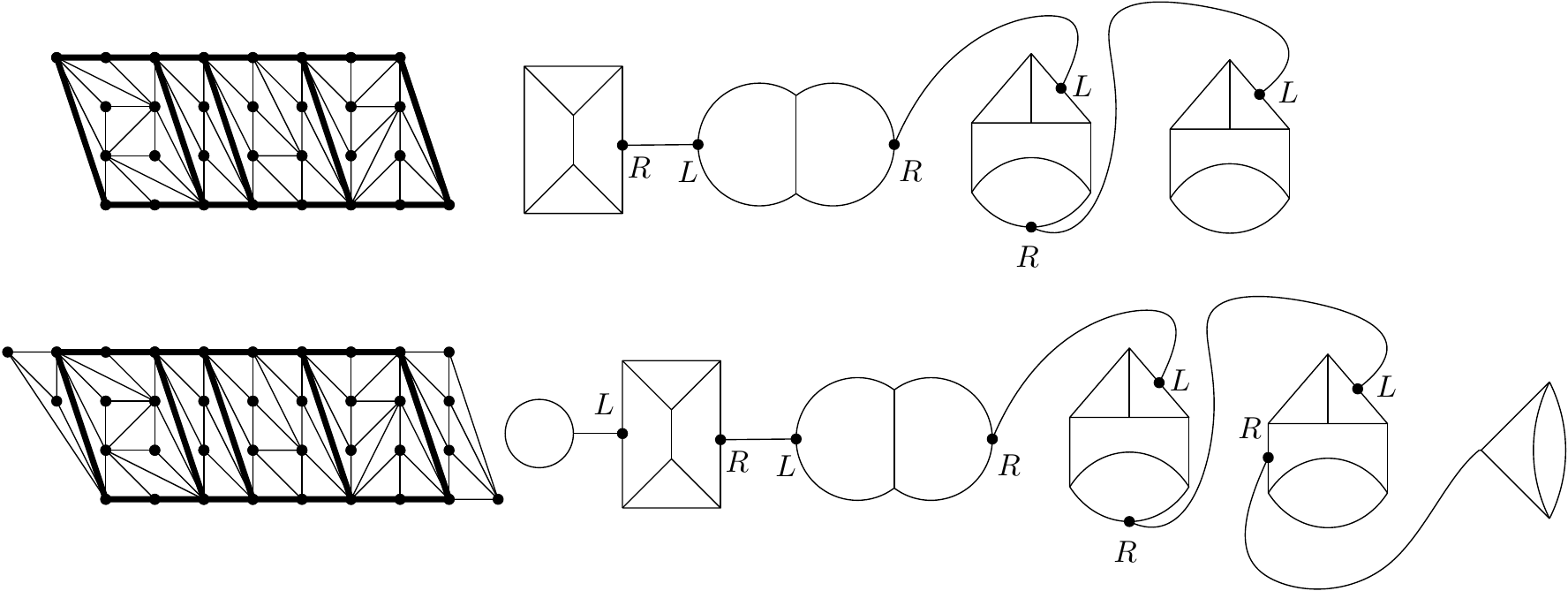}
	\caption{A tiling of $P^{||}_{14}$, and the corresponding troplanar graph; followed by a tiling of $Q_{7}^{\textrm{odd}}$ }
	\label{figure:tiled}
\end{figure}

For an integer $n$ let $Q_{n}^{\textrm{odd}}$ be the trapezoid with vertices at    $(0,3)$, $(2,0)$, $(n+2,3)$, and $(n+3,0)$.  This polygon has genus $2n+3$, and consists of a copy of $P^{||}_{2n}$ with two polygons glued on, one of genus $1$ and one of genus $2$.  Similarly, let $Q_{n}^{\textrm{even}}$ be the pentagon with vertices at $(0,1)$, $(0,3)$, $(2,0)$, $(n+2,3)$, and $(n+3,0)$.  These polygons are illustrated in Figure \ref{figure:odd_and_even}.

\begin{figure}[hbt]
   		 \centering
		\includegraphics[scale=0.8]{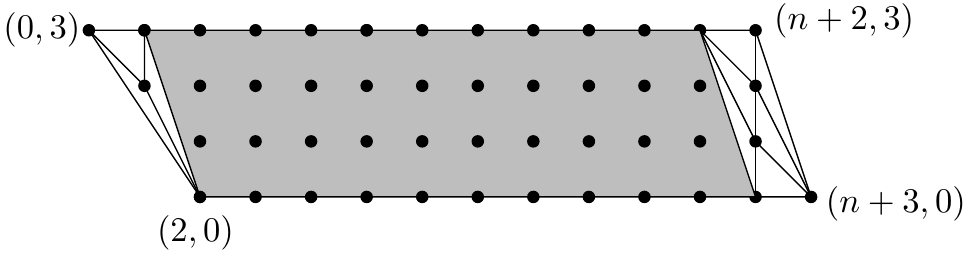}\quad\quad	\includegraphics[scale=0.8]{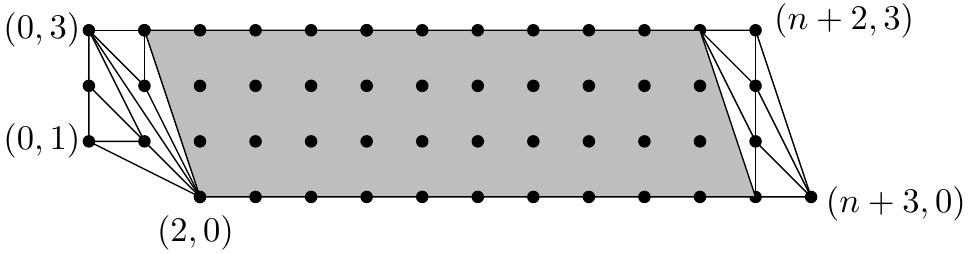}
	\caption{The polygons $Q_{n}^{\textrm{odd}}$ and $Q_{n}^{\textrm{even}}$, partially triangulated}
	\label{figure:odd_and_even}
\end{figure}

\begin{prop}\label{prop:same_sequence}  Let $\mathcal{T}$ and $\mathcal{T}'$ be triangulations of the same polygon, either  $Q_{n}^{\textrm{odd}}$ or $Q_{n}^{\textrm{even}}$, constructed by completing the partial triangulations in Figure \ref{figure:odd_and_even} by tiling the remaining copy of $P^{||}_{2n}$ with some sequences of our $1$ tile of genus $2$, our $8$ tiles of genus $4$, and our $49$ tiles of genus $6$.  Then $\mathcal{T}$ and $\mathcal{T}'$ are regular, and the troplanar graphs  arising from them are isomorphic if and only if they were constructed using the exact same sequence of tiles.
\end{prop}

\begin{proof}
Assume for the moment that we are triangulating the polygon $Q_{n}^{\textrm{odd}}$.  Let $\mathcal{T}$ be constructed using the sequence of tiles $(T_1,\cdots,T_k)$, and let $\mathcal{T}'$ be constructed using the sequence of tiles $(T'_1,\cdots,T'_m)$. Our triangulations $\mathcal{T}$ and $\mathcal{T}'$ are both regular, since they arise from regular triangulations glued along shared edges of lattice length $1$.

Let $G$ and $H$ be the troplanar graphs arising from $\mathcal{T}$ and $\mathcal{T}'$. Certainly if they were constructed from the same sequence of tiles, we have that $G$ is isomorphic to $H$.  Assume now that $G$ is isomorphic to $H$.  By construction, both $G$ and $H$ are of the form in Proposition \ref{prop:isomorphic}, where $G_1,\cdots,G_k$ are the graphs arising from the tiles $T_1,\cdots,T_k$, and $H_1,\cdots,H_m$ are the graphs arising from the tiles $T_1',\cdots,T_k'$.  By Proposition \ref{prop:isomorphic}, $k=m$ and we have that for all $i$, each $G_i$ is isomorphic to $H_i$ via an isomorphism respected the left and right marked points.  Thus for all $i$, $T_i$ and $T_i'$ give rise to isomorphic graphs, including the marked points.  By construction, each tile gives rise to a different marked graph, so $T_i=T_i'$ for all $i$.  We conclude that $G$ and $H$ were constructed from the exact same sequence of tiles.

This argument carries over to the polygon $Q_{n}^{\textrm{even}}$, with the additional footnote that $G_1$ and $H_1$ are both a biedge coming from the lattice point $(1,2)$ in $Q_{n}^{\textrm{even}}$.

\end{proof}

Thus in order to find a lower bound on the number of troplanar graphs of genus $g+3$ or $g+4$, we can count the number of ways to tile the parallelogram $P^{||}_g$ with our tiles of genus $2$, $4$, and $6$.  Before we do this, however, we will argue that we can actually include even more tiles than presented so far, in particular ones that yield graphs with bridges.

We will add one additional tile of genus $2$, and five additional tiles of genus $4$; these are pictured in Figure \ref{figure:genus24_bridges}, along with their marked troplanar graphs.  We will also add  $26$ tiles of genus $6$, illustrated in Appendix \ref{appendix}.  We have selected these tiles so that (when taking markings into consideration) no two give the same ordered pair of graphs, and so that no $2$-edge-connected component they contribute is available from a single tile.  (This is not immediately obvious for the final tile of genus $4$ pictured; however, the location of the bridge relative to the marked points cannot be obtained from our tiles of genus $2$.)

\begin{figure}[hbt]
   		 \centering
		\includegraphics[scale=1]{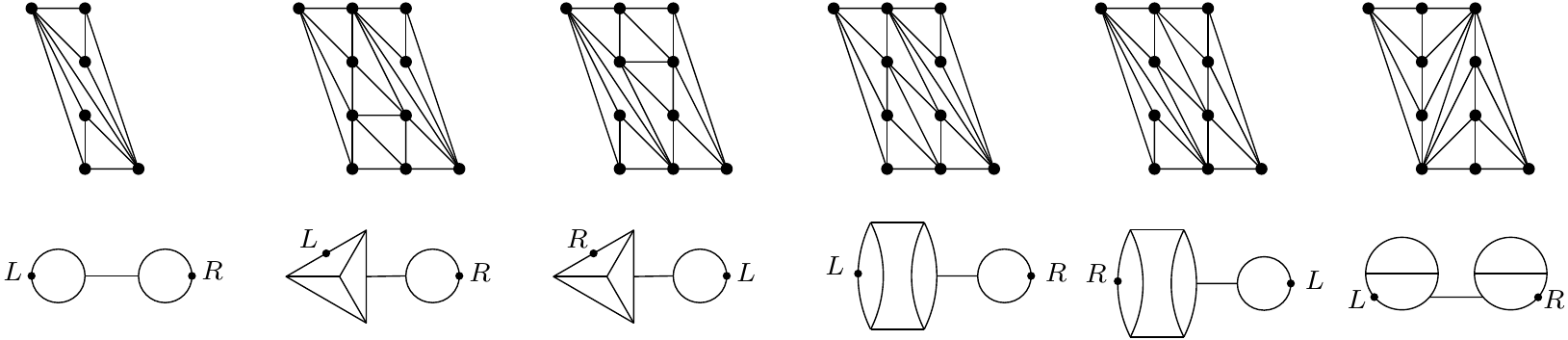}
	\caption{The tiles of genus $2$ and $4$ giving graphs with bridges}
	\label{figure:genus24_bridges}
\end{figure}

\begin{prop} \label{prop:same_sequence2} The result from Proposition \ref{prop:same_sequence} still holds with our additional tiles.
\end{prop}

\begin{proof}  We will assume we are triangulating the polygon $Q_{n}^{\textrm{odd}}$; a similar argument holds for $Q_{n}^{\textrm{even}}$.

Let $\mathcal{T}$, $\mathcal{T}'$, $G$, and $H$ be as in Proposition \ref{prop:same_sequence}, where the sequences of tiles for $\mathcal{T}$ and $\mathcal{T}'$ are $(T_1,\cdots,T_r)$ and $(T_1',\cdots,T_s')$.  It is still the case that $G$ and $H$ are of the form prescribed in Proposition \ref{prop:isomorphic}; however, it might now be that a tile $T_i$ (or $T_i')$ contributes more than one $2$-edge-connected component to $G$ (or $H$).  That is, we might have $r<k$ or $s<m$.

Assume that $G$ is isomorphic to $H$, and suppose for the sake of contradiction that $(T_1,\cdots,T_r)\neq(T_1',\cdots,T_s')$.  Since the two sequences must contribute a total of genus $2n$, they must differ in some $t^{th}$ entry where $1\leq t\leq\min\{r,s\}$.  Let $i$ be the first index for which $T_i\neq T_i'$.  Since $(T_1,\cdots,T_{i-1})=(T_1',\cdots,T_{i-1}')$, isomorphism between $G$ and $H$ is preserved if we contract all portions of $G$ and $H$ arising from the first $i-1$ tiles; thus, we may assume without loss of generality that $i=1$.

Since $G$ is isomorphic to $H$, we know that $G_1$ is isomorphic to $H_1$ and $G_2$ is isomorphic to $H_2$, with the isomorphisms respecting the left and right marked points.  The graphs $G_1$ and $H_1$ must come from $T_1$ and $T_1'$.  Since $T_1\neq T_1'$, at least one of these tiles, say $T_1$, must contribute another $2$-edge-connected component to $G$; thus $G_2$ also comes from $T_1$.  However, by construction no tile contributing two $2$-edge-connected components contributes a $2$-edge-connected component that is available from a tile contributing only one $2$-edge-connected component.  It follows that $T_1'$ must also contribute $H_2$.  This is all $T_1$ and $T_1'$ can contribute, since no tile contributes more than two $2$-edge-connected components.  However, no distinct pair of our tiles $T_1$ and $T_1'$ give the same ordered pair of marked $2$-edge-connected components, a contradiction.  We conclude that $(T_1,\cdots,T_r)=(T_1',\cdots,T_s')$.

\end{proof}

Let $a_n$ denote the number of ways to tile the parallelogram $P^{||}_{2n}$ with our tiles.  Our propositions imply that each of the $a_n$ tilings gives a distinct troplanar graph of genus $2n+3$, and a distinct troplanar graph of genus $2n+4$. Thus, we have $\mathscr{T}(g)\geq a_{\lfloor (g-3)/2\rfloor}$.  To obtain an explicit lower bound on $\mathscr{T}(g)$, we will determine a formula for $a_n$.

\begin{prop}
\label{prop:tilingcount}
Let $a_n$ be the number of ways to tile the parallelogram  $P^{||}_{2n}$ of genus $g=2n$ with our tiles of genus $2$, $4$ and $6$.  Then $a_n=\Omega(\alpha^n)$, where $\alpha\approx 6.1233$ is the unique real root of  $x^3-2x^2-13x-75$.
\end{prop}

\begin{proof}

Any tiling of $P^{||}_{2n}$ will either begin with a tile of genus $2$, followed by a tiling of $P^{||}_{2(n-1)}$; or with a tile of genus $4$, followed by a tiling of $P^{||}_{2(n-2)}$; or with a tile of genus $6$, followed by a tiling of $P^{||}_{2(n-3)}$.  It follows that $a_n$ satisfies the recurrence relation
\[a_n=2a_{n-1}+13a_{n-2}+75a_{n-3}\]
for $n\geq 3$.  The initial conditions are $a_0=1$ (there is one way to tile the empty parallelogram, namely not to tile anything), $a_1=2$ (since there are two tiles of genus $2$), and $a_2=17$ ($13$ from the tiles of genus $4$, and $4$ from the $4$ ways to place two tiles of genus $2$).

The solution to this recurrence relation can be found by studying the roots of the characteristic polynomial $x^3-2x^2-13x-75$.  This polynomial has three distinct roots:  a real root $\alpha\approx 6.1233$, and two complex conjugate roots $\beta= r e^{i\theta}$ and $\overline{\beta}= r e^{-i\theta}$, where $r\approx 3.4998$ and $\theta\approx 2.2007$ (in radians). Any (complex) solution to the recurrence relation has the form\[a_n=A\alpha^n+B\beta^n+C\overline{\beta}^n,\]
where $A,B,C\in\mathbb{C}$.  
Our initial conditions imply 
\[A+B+C=1,\]
\[A\alpha+B\beta+C\overline{\beta}=2,\]
\[A\alpha^2+B\beta^2+C\overline{\beta}^2=17,\]
which we can rewrite as
\[\left(\begin{matrix}1&1&1\\\alpha&\beta&\overline{\beta}\\\alpha^2&\beta^2&\overline{\beta}^2\end{matrix}\right)\left(\begin{matrix}A\\B\\C\end{matrix}\right)=\left(\begin{matrix}1\\2\\17\end{matrix}\right).\]
Numerically solving this system of equations using Mathematica \cite{Mathematica}, we find  $A\approx 0.49999$, $B\approx 0.25001+0.00543i$, and $C=\overline{B}$.  We can then rewrite our solution as
$$a_n=A\alpha^n+2Dr^n\cos(n\theta-\delta),$$
where $D=|B|$ and $\delta=\textrm{arg}(B)$. 
Since  $A>0$ and $\alpha>r$, we have that $|2Dr^n\cos(n\theta-\delta)|\leq \frac{A}{2}\alpha^n$ for sufficiently large $n$.  This implies that \[a_n=A\alpha^n+2Dr^n\cos(n\theta-\delta)\geq \frac{A}{2}\alpha^n\]
for sufficiently large $n$.  Since $\frac{A}{2}>0$, this means that $a_n=\Omega(\alpha^n)$.

\end{proof}

This allows us to deduce the following asymptotic lower bound on $\mathscr{T}(g)$.

\begin{cor}\label{cor:upper} We have $\mathscr{T}(g)=\Omega(\gamma^g)$ as $g\rightarrow\infty$, where $\gamma=\sqrt{\alpha}\approx 2.47$ is the the square-root of the unique real root of $x^3-2x^2-13x-75$.
\end{cor}

\begin{proof}
By Propositions \ref{prop:same_sequence} and \ref{prop:same_sequence2}, we have have $\mathscr{T}(g)\geq a_{\lfloor(g-3)/2\rfloor} $ for all $g\geq 5$.  Since $a_n=\Omega(\alpha^n)$ by Proposition \ref{prop:tilingcount}, we have $a_{\lfloor(g-3)/2\rfloor}=\Omega(\alpha^{\lfloor(g-3)/2\rfloor})=\Omega(\alpha^{g/2})=\Omega(\sqrt{\alpha}^g)$.  We conclude that $\mathscr{T}(g)=\Omega(\sqrt{\alpha}^g)$.
\end{proof}

We close our paper with several possible directions for future research.  Our upper and lower bounds are roughly  exponential in $g$, with bases around $2^{11/3}\approx 12.699$ and $\gamma\approx 2.47$, respectively; closing the gap between these bases is a natural goal for future work. For a better lower bound, the methods of Section \ref{sec:lower} could be pushed further, perhaps by considering tiles of higher genus. One approach for a better upper bound, in the spirit of Proposition \ref{prop:width_3}, could be to  bound $\mathscr{T}^{(2)}(g)$, the number of $2$-edge-connected troplanar graphs.  If one proved a bound of the form  $\mathscr{T}^{(2)}(g)=O(b^g)$, then Corollary \ref{corollary:two-connected} would imply $\mathscr{T}(g)=O((2b)^g)$.  Thus to improve Theorem \ref{theorem:best}, it would suffice to show $\mathscr{T}^{(2)}(g)=O(b^g)$ for some $b< 2^{11/3}/2\approx 6.3496$.

In general there is no easily implemented single criterion (or list of such criteria) for determining whether a graph is troplanar, and finding new criteria could be helpful in furthering our understanding of such graphs. For instance, although many results (such as Proposition \ref{prop:sprawling} and Theorem \ref{theorem:tie}) forbid certain structures that cannot be dual to general unimodular triangulations, no existing work has specifically used the \emph{regularity} of triangulations to forbid certain structures.  In fact, it is unknown whether considering dual graphs of nonregular triangulations yields any non-troplanar graphs. There is also no known example  of a $2$-edge-connected trivalent non-troplanar graph that is neither nonplanar nor crowded.  Finding such an example could point to  other general obstructions to troplanarity.

There are other reasonable families of graphs one could study that get at the notion of being ``tropically planar.'' In the notation of \cite{BJMS}, the moduli space of all metric graphs of genus $g$ arising from smooth tropical plane curves is written $\mathbb{M}_g^{\textrm{planar}}$.  Asking if a graph $G$ is tropically planar, then, is asking whether $G$ appears (for at least one choice of edge lengths) in $\mathbb{M}_g^{\textrm{planar}}$.  However, one could also study $\textrm{trop}(\mathcal{M}_g^{\textrm{nd}})$, the tropicalization of the moduli space of nondegenerate curves studied in \cite{cv}.  This is the space of all metric graphs of genus $g$ that arise as the Berkovich skeleton of a nondegenerate curve, and strictly contains $\mathbb{M}_g^{\textrm{planar}}$ for $g\geq 3$.  There are certainly combinatorial types of graphs that appear in $\textrm{trop}(\mathcal{M}_g^{\textrm{nd}})$ that do not appear in $\mathbb{M}_g^{\textrm{planar}}$; for instance, all graphs of genus $3$ and $4$ appear in $\textrm{trop}(\mathcal{M}_3^{\textrm{nd}})$ and $\textrm{trop}(\mathcal{M}_4^{\textrm{nd}})$, respectively, even though not all graphs of genus $3$ and $4$ are tropically planar.  Even for $g=5$, determining which combinatorial types of graphs appear in $\textrm{trop}(\mathcal{M}_g^{\textrm{nd}})$ appears to be an open question.

There are also alternate definitions of what is meant by a ``tropical plane curve.''  Instead of considering tropical curves that are a subset of $\mathbb{R}^2$, one could more generally consider tropical curves that are subsets of $2$-dimensional tropical linear spaces embedded in some $\mathbb{R}^n$.  The authors of \cite{hmrt} take this approach, and show that every metric graph of genus $3$ (with a low-dimensional set of exceptions) arises in such a $2$-dimensional tropical linear space in $\mathbb{R}^n$ for $n\leq 5$; this includes metric graphs whose underlying graph is the non-troplanar graph of genus $3$.  A natural question to ask is which graphs of genus $g$ appear in tropical curves embedded in such tropical linear spaces for $g\geq 4$.

\begin{appendices}
\section{Tiles of genus 6} \label{appendix}

In this appendix we present the tiles of genus $6$ used in our construction of troplanar graphs in Section \ref{sec:lower}.  The $49$ tiles yielding bridgeless graphs are pictured in Figure \ref{figure:tiles_g6}.   The $26$ tiles yielding  graphs with bridges are pictured in Figure \ref{figure:bridge_tiles_g6}.  In both figures, we box together isomorphism classes of graphs, which are distinguished by the marked $L$ and $R$ points.

\begin{figure}[hbt]
   		 \centering
		\includegraphics[width=\textwidth]{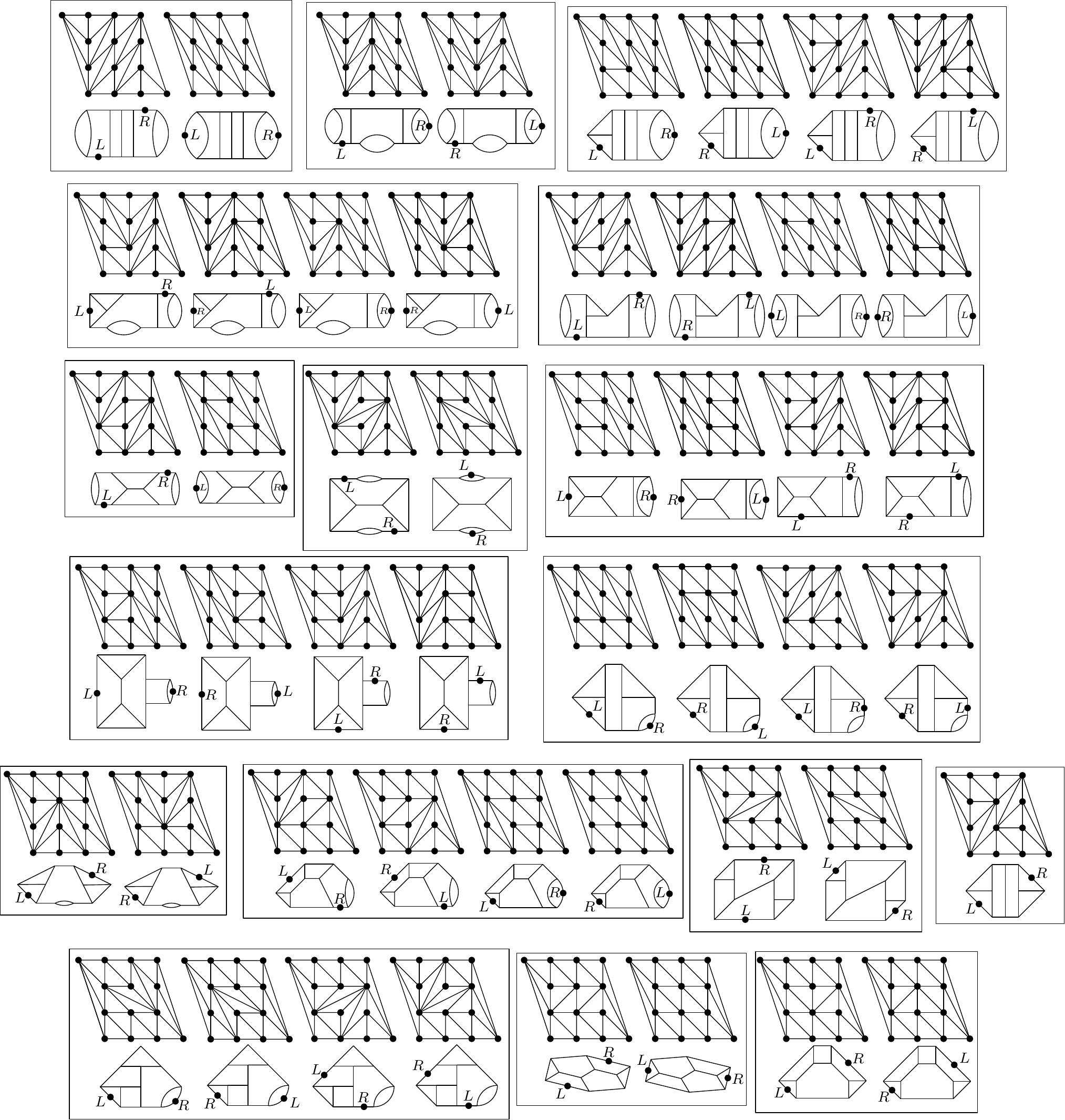}
	\caption{The tiles of genus $6$ giving graphs without bridges}
	\label{figure:tiles_g6}
\end{figure}

\begin{figure}[hbt]
   		 \centering
		\includegraphics[width=\textwidth]{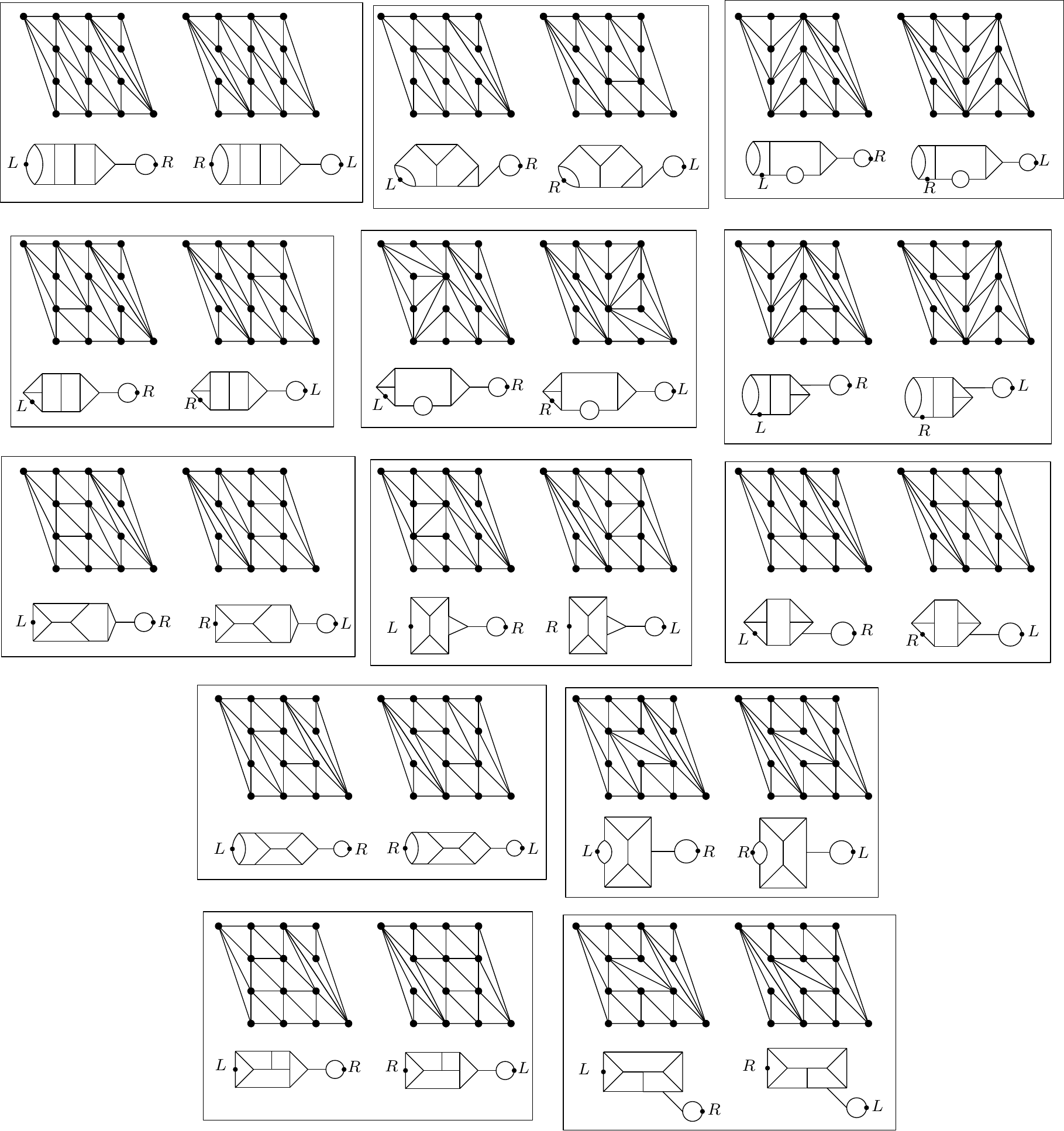}
	\caption{The tiles of genus $6$ giving graphs with bridges}
	\label{figure:bridge_tiles_g6}
\end{figure}

\end{appendices}




\bibliographystyle{abbrv}

\end{document}